\documentclass[11pt,letterpaper]{article}
\usepackage[T1]{fontenc}
\usepackage[utf8]{inputenc}  
\usepackage{amsmath}
\usepackage{amssymb}
\usepackage{amsthm}
\usepackage{listings}
\usepackage{verbatim}
\usepackage{placeins}

\usepackage[variablett]{lmodern}
\usepackage{xcolor}
\makeatletter

\makeatletter
\newcommand\BeraMonottfamily{%
  \def\fvm@Scale{0.85}
  \fontfamily{fvm}\selectfont
}
\makeatother

\usepackage{amsfonts}
\usepackage{enumitem}
\usepackage{bm}
\usepackage{cancel}
\usepackage{thmtools}
\usepackage[capitalise]{cleveref}
\usepackage{multirow}

\usepackage[margin=1in]{geometry}

\usepackage{xspace}
\usepackage{dot2texi}

\usepackage{ifthen}
\usepackage{pgf,tikz}
\usetikzlibrary{arrows,automata,shapes}
\usetikzlibrary{patterns,fadings}
\usetikzlibrary{fit}
\tikzstyle{every state}=[minimum size=12pt,inner sep=0pt]
\usetikzlibrary{snakes}
\tikzstyle{randomPath}=[decorate,decoration={amplitude=1pt,segment length=2pt,random steps},very thick]
\tikzstyle{randomPath2}=[decorate,decoration={amplitude=2pt,segment length=6pt,random steps},very thick]
\usepackage[]{algorithm2e}
\usetikzlibrary{automata,shapes,snakes,positioning,shapes.geometric,calc}

\usepackage{wrapfig}

\usepackage{authblk}

\renewcommand{\a}{\alpha}

\renewcommand{\d}{\delta}
\newcommand{\D}{\Delta}

\renewcommand{\l}{\lambda}

\renewcommand{\O}{\Omega}

\newcommand{\s}{\sigma}
\renewcommand{\SS}{\Sigma}

\newcommand{\F}{\mathbb F}

\newcommand{\R}{{\mathbb R}}

\newcommand{\Z}{{\mathbb Z}}

\newcommand{\Int}{{\rm Int}}
\newcommand{\Sym}{{\rm Sym}}

\newcommand{\tar}{{\rm tar}}
\newcommand{\aux}{{\aux}}

\newtheorem{theorem}{Theorem}[section]
\newtheorem{proposition}[theorem]{Proposition}
\newtheorem{lemma}[theorem]{Lemma}
\newtheorem{corollary}[theorem]{Corollary}

\newtheorem{fact}[theorem]{Fact}

\theoremstyle{definition}
\newtheorem{definition}[theorem]{Definition}
\newtheorem{assumption}[theorem]{Assumption}

\theoremstyle{remark}
\newtheorem{example}[theorem]{Example}

\newcommand{\til}[1]{{\widetilde{#1}}}
\newcommand{\prop}{{\rm P}}

\newcommand{\aut}[1]{{\mathcal #1}}
\newcommand{\auta}{\aut{A}}
\newcommand{\mot}[1]{{\mathbf {#1}}}

\newcommand{\pres}[1]{\langle{#1}\rangle}
\newcommand{\presp}[1]{\pres{{#1}}_{+}}
\newcommand{\gauta}{\pres{\auta}}

\newcommand{\sect}[2]{{{#1}_{|#2}}}
\newcommand{\act}[2]{{#1}.{#2}}
\newcommand{\perm}[1]{\pi{\left(#1\right)}}
\newcommand{\wrpr}[2]{{\left\langle#1\right\rangle}{#2}}
\newcommand{\wrprp}[3]{{{\left\langle#1\right\rangle}^{#2}}{#3}}

\newcommand{\QQ}{Q}
\newcommand{\XX}{\Sigma}
\newcommand{\auttuple}{\left(\QQ,\XX,\delta,\rho\right)}

\newcommand{\grgrow}{\gamma}

\usetikzlibrary{lindenmayersystems}
\pgfdeclarelindenmayersystem{cayley}{
  \rule{F -> F [ R [F] [+F] [-F] ]}
  \symbol{R}{
    \pgflsystemstep=0.5\pgflsystemstep
  } 
}

\begin{document}


\date{}

\title{Numerical upper bounds on growth of automata groups}

\author[1]{J\'er\'emie Brieussel}

\author[1,2,3]{Thibault Godin}

\author[1]{Bijan Mohammadi}

\affil[1]{Universit\'e de Montpellier -- Institut Montpelli\'erain Alexander Grothendieck}

\affil[2]{Universit\'e de Lorraine -- Institut Elie Cartan de Lorraine}

\affil[3]{University of Turku -- Department of Mathematics}


\maketitle

\begin{abstract}
The growth of a finitely generated group is an important geometric invariant which has been studied for decades. It can be either polynomial, for a well-understood class of groups, or exponential, for most groups studied by geometers, or intermediate, that is between polynomial and exponential. Despite recent spectacular progresses, the class of groups with intermediate growth remains largely mysterious. Many examples of such groups are constructed using Mealy automata. The aim of this paper is to give an algorithmic procedure to study the growth of such automata groups, and more precisely to provide numerical upper bounds on their exponents.

  Our functions retrieve known optimal bounds on the famous first Grigorchuk group. They also improve known upper bounds on other automata groups and permitted us to discover several new examples of automata groups of intermediate growth. All the  algorithms described are implemented in~\texttt{GAP}, a language dedicated to computational group theory.
\end{abstract}

\section{Introduction}

The aim of this paper is to develop tools to analyse the growth of automata groups, and more precisely to provide numerical upper bounds when the growth is intermediate.

Groups are  fundamental objects in mathematics, which can be seen as an abstract encoding of the notion of symmetry. They are used in various domains of discrete mathematics and combinatorics, e.g.  graph isomorphism~\cite{Hel18}), cryptography~\cite{Crypto}, physics~\cite{Physics}, chemistry~\cite{Chimie} and even biology~\cite{EFV14}. The ideal goal of group theory is to describe and classify all the possible behaviours that a group can exhibit. A landmark result in this direction is the~\emph{classification of finite simple groups}~\cite{Atlas85}. For infinite groups, various interesting phenomena can appear. For instance a group can be infinite and yet have all its elements of finite order. Such groups are called Burnside groups in homage to Burnside who asked whether they exist in 1902. First examples were given in 1964 by Golod and Shafarevich~\cite{Gol64,GoSh64}. Subsequent examples were given in 1968 by Novikov and Adyan~\cite{NA68} and in 1972 by Aleshin~\cite{Ale72}, the later using automata groups. 

An geometric  way to classify  groups is through their growth: let \(G\) be a finitely generated group and \(S\) a generating set of \(G\). The \emph{Cayley graph} of \((G,S)\) is the graph \(\Gamma_{G,S} \) whose vertices are the elements of \(G\) and such that there is an oriented edge from \(g\) to \(h\) if there exists \(s \in S\) such that \(g.s=h\). Such an edge is naturally labelled by $s$. Examples of such graphs are depicted in Fig.~\ref{fig-Cayley}.
The \emph{growth} function of a group \(G\) with respect to \(S\) is the function $\gamma_{G,S}(\ell)$
that counts the number of elements in a ball of given radius~\(\ell\) in~ \(\Gamma_{G,S} \). This notion was introduced by Svarc~\cite{Svarc} and Milnor~\cite{Mil68int} in relation with Riemannian geometry. They observed that the growth of a group is essentially independent of the generating set (see Lemma~\ref{lem:growthindegen}). Obviously the growth is bounded if and only if the group is finite. Moreover abelian (i.e. commutative) groups have polynomial growth. For instance, integer lattices $\Z^d$ have polynomial growth of degree $d$. This is still the case for the (slightly bigger) class of nilpotent groups~\cite{Wolf68}, and a celebrated result of Gromov asserts that there are essentially no other such groups, since polynomial growth implies nilpotency up to taking a finite index subgroup~\cite{Gro81}. 

\begin{figure}[h!]
\scalebox{0.6}{
	\begin{tikzpicture}[->,>=latex,node distance=25mm,scale=1]
	
	\node[black] (1) [] {$(0,0)$};
  	\node[black]  (2) [right of=1] {$(1,0)$};
  	\node[black]  (3) [right of=2] {$(2,0)$};
  	\node[black]  (4) [above of=1] {$(0,1)$};
   	\node[black]  (5) [above of=4] {$(0,2)$}; 	

  	\node[black]  (-2) [left of=1] {$(-1,0)$};
  	\node[black]  (-3) [left of=-2] {$(-2,0)$};
  	\node[black]  (-4) [below of=1] {$(0,-1)$};
   	\node[black]  (-5) [below of=-4] {$(0,-2)$}; 	
   	
   	\node[black]  (a) [right of=4] {$(1,1)$};
  	\node[black]  (b) [below of =2] {$(1,-1)$};
  	\node[black]  (c) [left of=4] {$(-1,1)$};
   	\node[black]  (d) [below of=-2] {$(-1,-1)$};

	\path 
	(1) edge node[above]{$a$}(2)
	(2)  edge node[above]{$a$} (3)
	(1) edge node[above]{$a^{-1}$} (-2)
	(-2) edge node[above]{$a^{-1}$}(-3)
	(4) edge node[above]{$a$} (a)
	(4) edge node[above]{$a^{-1}$} (c)
	(-4) edge node[above]{$a$} (b)
	(-4) edge node[above]{$a^{-1}$}(d)

	(1) edge node[right]{$b$} (4)
	(1) edge node[right]{$b^{-1}$} (-4)
	(4) edge node[right]{$b$} (5)
	(-4) edge node[right]{$b^{-1}$}(-5)
	
	(2) edge node[right]{$b$} (a)
	(-2) edge node[right]{$b^{-1}$} (c)
	(2) edge node[right]{$b^{-1}$} (b)
	(-2) edge node[right]{$b$}(d)
	
	;
	

\draw[thick,dotted,-] (2.5,2.8) -- (2.5,4.25);	
\draw[thick,dotted,-] (3.1,2.5) -- (4.25,2.5);		
\draw[thick,dotted,-] (-2.5,2.8) -- (-2.5,4.25);	
\draw[thick,dotted,-] (-3.1,2.5) -- (-4.25,2.5);		
\draw[thick,dotted,-] (2.5,-2.8) -- (2.5,-4.25);	
\draw[thick,dotted,-] (3.1,-2.5) -- (4.25,-2.5);		
\draw[thick,dotted,-] (-2.5,-2.8) -- (-2.5,-4.25);	
\draw[thick,dotted,-] (-3.1,-2.5) -- (-4.25,-2.5);		

{\draw[red,thick,dashed,-] (3.25,0)--(0,3.25)--(-3.25,0)--(0,-3)--(3.25,0);}

\node[black] (gamma1) at (1,1.75) {\color{red}$\gamma(1)$};

{\draw[blue,thick,dashed,-] (5.75,0)--(0,5.75)--(-5.75,0)--(0,-5.75)--(5.75,0);}

\node[black] (gamma2) at (3.25,-3.25) {\color{blue}$\gamma(2)$};

\end{tikzpicture}
	}
	\hspace{0.8cm}
	\scalebox{0.35}{
\begin{tikzpicture}
\draw l-system [l-system={cayley, axiom=[F] [+F] [-F] [++F], step=5cm, order=4}];
\end{tikzpicture}
}

	\caption{On the left: the Cayley graph of \(\Z^2\) with generators \(a^{\pm1}=(0,\pm1)\) and \(b^{\pm1}=(\pm1,0)\) together with the balls of radii \(1\) et \(2\). On the right: the ball of radius \(5\) in the (unlabelled) Cayley graph of the free group on 2 generators.}\label{fig-Cayley}
\end{figure}
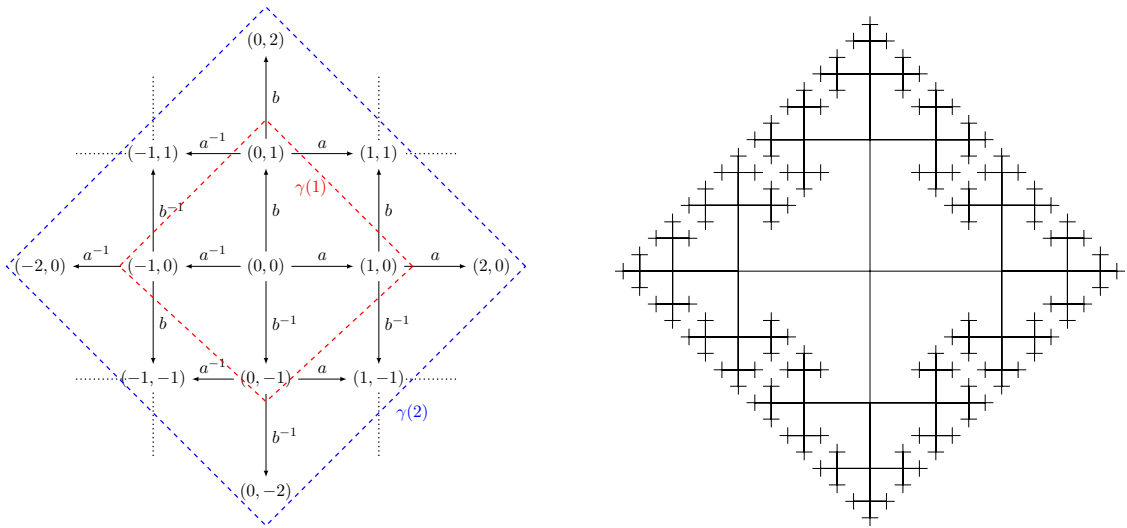

On the other hand, free groups, whose Cayley graphs are essentially infinite regular trees, and non-elementary hyperbolic groups have exponential growth. Among several classes of groups, the growth is either polynomial or exponential. This dichotomy holds for instance among solvable groups by Milnor and Wolf~\cite{Mil68Note,Wolf68}, or among linear groups, i.e. groups of invertible matrices, by the Tits alternative~\cite{Tits72}. We refer the reader to more complete introduction to this topic in~\cite{dlHar00,GrPa08,Man11,Gri14}. 

Motivated by such observations, Milnor asked in 1968 if this picture was complete or if there exists groups having growth functions between polynomial and exponential, henceforth called \emph{groups of intermediate growth}~\cite{Mil68}. This question was positively answered by Grigorchuk in the eighties:

\begin{theorem}[\cite{Gri83,Gri85},\cite{Bar98,ErZh18}]
	The group generated by the Mealy automaton of \cref{fig:Mealy} has intermediate growth. Its growth function is essentially equivalent to~\(\exp(\ell^{0.7674\dots})\).
\end{theorem}

The computation of the growth function of this group is challenging. Grigorchuk proved that it is between $e^{\sqrt{n}}$ and $e^{\ell^{0.993}}$ in \cite{Gri85}. Bartholdi gave an accurate upper bound~\cite{Bar98}. Only recently Erschler and Zheng provided a matching (up to logarithmic factor) lower bound~\cite{ErZh18}.

The group generated by the Mealy automaton depicted in \cref{fig:Mealy} is now called the \emph{first Grigorchuk group}.  Mealy automata are a special class of transducers (automata with input and output) whose functions induced by states on words generate groups. These have been ubiquitous in growth related questions since the eighties: most constructions of groups of intermediate growth use them~\cite{Gri85,Grigorchuk1986,FaGu91,BaSu01,BartholdiNonU2003,Brieussel2009,BartholdiErschler2012,Bri17} they are a basic tool to construct groups of given intermediate growth~\cite{BaEr14,Bri14} and they were also used by Wilson~\cite{Wil04} to exhibit groups having non-uniform exponential growth (answering a question of Gromov). Only recently Nekrashevych gave a construction of  groups of intermediate growth not based on the use of Mealy automata~\cite{Nek18}.

	This new construction is very different since Nekrashevych's groups are \emph{simple} (i.e. have no non-trivial normal subgroups) whereas automata groups are residually finite (in particular, they have infinitely many normal subgroups of finite index).

 In spite of this list of examples, few things are actually known about the growth of general automata groups: some restricted classes of automata generate only groups with bounded or exponential growth~\cite{Kli17,Olu17}. It contrasts with the usual amenability of automata groups to algorithmic questions~\cite{Bar16,AKLMP12,KMP12,BGKP18,GAP,FR,AutomGrp}.
\begin{definition} A number $\a \in [0,1]$ is called an \emph{upper bound exponent} of the group $G$ if for any generating set $S$ there exists $c>0$ such that $\gamma_{G,S}(\ell) \leq \exp(c\ell^\alpha)$ for all $\ell$.
\end{definition}

It is an open question whether there exists groups with growth exponent smaller than $0.7674...$. Grigorchuk's gap conjecture asserts that there are no groups with growth below $\exp(\sqrt{\ell})$. So far, only the groups in \cite{Gri85,Grigorchuk1986,BaSu01,Bri17} (and related constructions) are known to have an exponent $<1$. It is likely that the groups in \cite{FaGu91,BartholdiNonU2003,Brieussel2009} have intermediate growth but exponent $1$ as their growth rate could be essentially $\exp(\frac{\ell}{\log(\ell)})$. Such a behavior is known to hold for the group $G_{(01)^\infty}$ also introduced by Grigorchuk~\cite{Gri85} and studied by Erschler~\cite{Erschler2004boundary}. Nekrashevych conjectures that some groups in \cite{Nek18} have growth exponent below $0.7674...$, but it is not yet proved that they have exponent $<1$. Regarding automata groups, the only accurate computations of exponents concern Grigorchuk groups acting on the binary tree. They appeared in a recent work of Erschler and Zheng~\cite{ErZh18}. 

The present work aims at giving explicit numerical upper bounds on the growth of automata groups and to yield new examples of such groups.\\

{\bf Proposed approach:}
All known examples of automata groups of intermediate growth belong to the class of \emph{(strongly) contracting} automata. More precisely, the contraction method described in Section \ref{sec:contcrit} and originaly due to Grigorchuk~\cite{Gri83,Gri85,Grigorchuk1986} is the only strategy known so far to obtain non-trivial upper bounds for growth exponents.

An element~\(g\) of an automaton group~\(\gauta\) acting on~\(\XX^*\) can be recursively described  as~\(g=\wrpr{\sect{g}{x_1},\hdots, \sect{g}{x_d}}{\perm{g}}\) where~\(\perm{g}\) is a permutation of~\(\XX\) and the~\(\sect{g}{x}\) are elements of the group called the~\emph{sections} of~\(g\). We say that~\(g\) is \emph{(strongly) contracting} if the length of~\(g\) is strictly smaller than the sum of the length of its section, and we set~\(\eta(g)\) to be their ratio. 

The contraction method ensures that if we find a suitable finite set~\(\D\) containing only contracting elements, the group has subexponential growth bounded above by~\(\exp(\ell^\alpha)\), where~$\alpha$ depends explicitly on the ratio~$\eta$ and the size of the alphabet. Suitable sets have to be \emph{essentially geodesic generating}, in short eggs. This property is defined in Section~\ref{sec:eggs}.

In practice, given a target $\eta$, we try to find such an egg~\(\D\). For this, we focus on an over-approximation of~\(\gauta\) in the sense that we perform our computations in a group~\(\tilde{G}\) with the same generators as~\(\gauta\) but only a finite number of relations, so that~\(\gauta\) is a quotient of~\(\tilde{G}\), and explore the Cayley graph of the latter group. The basic idea is to implement a graph traversal from the identity, stopping the exploration of a branch when we have found an element which is contracting enough or which has already been visited. Using a tight over-approximation allows to accelerate computations but in balance we have to keep a good understanding of the Cayley graph structure in order to guarantee that the computed bound is still relevant in the original automaton group -- see \cref{sec:stabquo}. This first (semi-)algorithm already presents an interesting combinatorics as the results may vary substantially according to the target. 

In a second time, we introduce weights in computations, as it was successfully done by Bartholdi for the first Grigorchuk group~\cite{Bar98}. Weights are positive functions~$\pi$ on the generating set $S$, and we declare the length of an edge to equal the weight of the generator that labels it. This gives a new, non-uniform and possibly non-symmetric, metric on the Cayley graph. It is well-known  that the growth rate is not modified by the introduction of weights, nor by change of generating set -- see Lemma~\ref{lem:growthindegen}. The upper bound~$\alpha$ obtained by the contraction method depends drastically on the weight $\pi$  as can be seen on \cref{table:resultGri,table:resultNew}, but by invariance the result applies to all growth functions~$\gamma_{G,S,\pi}(\ell)$.

A crucial issue is to find the best possible weights for a given generating set. For this, we fix an essentially geodesic generating set $\Delta$ and optimize the returned contraction parameter $\eta(\pi)$ under the condition that $\Delta$ remains an egg for $\pi$, which is the case under mild triangular conditions described in \cref{sec:stabweight}. In practice, we are lead to minimize a function of the form $\max_{w\in \Delta}(f_w)$, where the functions $f_w$ are rational (homogeneous) functions on $\R^d$, on a subdomain of $[0,1]^d$ bounded by linear inequalities. For this, we use a numerical algorithm  \texttt{wmo} based on a generalized gradient method described in \cite{Moh07, MoRe09} -- see~\cref{sec:wmo}.

Finally, using the weight-metric-optimization algorithm \texttt{wmo}, we develop a dynamical procedure trying to optimize the contraction through both the graph-search of an egg $\Delta$ and the choice of the weights $\pi$. 
Notice that these two possible directions of optimization, by graph exploration and weights choice, strongly interact. Hence, a careful tuning of the parameters might be needed for some groups. On the other hand, pre-set routines often give rather good results-- see~\cref{table:resultGriopt,table:resultNew}.\\

\begin{table}[p]
\centering
	\begin{tabular}{|c|c|c|c|c|}
\cline{1-5}
level &weights chosen&target~\(\eta\) & \multicolumn{1}{ c| }{obtained~\(\eta\)} & obtained $\alpha$ 
\\ \cline{1-5}
\multicolumn{1}{ |c  }{\multirow{3}{*}{1} } &
\multicolumn{1}{ |c| }{\([.25,.25,.25,.25]\)} & \(.99\) & does not end & does not end 
\\ \cline{2-5}
\multicolumn{1}{ |c  }{}                        &
\multicolumn{1}{ |c| }{\([.25,.375,.25,.125]\)} & \(.99\) & .875 & .8385
\\ \cline{2-5}
\multicolumn{1}{ |c  }{}                       & 
\multicolumn{1}{ |c| }{\([.305061,.34747,.223839,.123631]^*\)} & \(.99\) & \(.8106^*\) &
\(.7675^*\)    
\\ \cline{1-5}

\multicolumn{1}{ |c  }{\multirow{4}{*}{2} } &
\multicolumn{1}{ |c| }{\([.25,.25,.25,.25]\)} & \(.99\) & does not end & does not end
\\ \cline{2-5}

\multicolumn{1}{ |c  }{}                        &
\multicolumn{1}{ |c| }{\([.25,.375,.25,.125]\)} & \(.99\) & ~\(.9231\) & \(.9455\)
\\ \cline{2-5}
\multicolumn{1}{ |c| }{} &
\multicolumn{1}{ |c| }{\([.305061,.34747,.223839,.123631]^*\)
} & \(.99\) & ~\(.9152\) & \(.9400
\)
\\ \cline{2-5}
\multicolumn{1}{ |c| }{} &\multicolumn{1}{ |c| }{\([.305061,.34747,.223839,.123631]^*\)} & \(.75\) & ~\(.7497\) & \(.8280\)
\\ \cline{2-5}
\multicolumn{1}{ |c| }{} &\multicolumn{1}{ |c| }{\([.305061,.34747,.223839,.123631]^*\)} & .68  & .6800  & .7824
\\ \cline{1-5}

\multicolumn{1}{ |c|  }{\multirow{3}{*}{3} }&
\multicolumn{1}{ |c| }{\([.25,.25,.25,.25]\)
} & \(.99\) & ~\(.8889\) & \(.9437
\)
\\ \cline{2-5}&
\multicolumn{1}{ |c| }{\([.305061,.34747,.223839,.123631]^*\)
} & \(.99\) & .8954  & .9496
\\ \cline{2-5} &
\multicolumn{1}{ |c| }{\([.305061,.34747,.223839,.123631]^*\)
} & \(.75\) & .7418  & .8745
\\ \cline{2-5} &
\multicolumn{1}{ |c| }{\([.305061,.34747,.223839,.123631]^*\)
} & \(.58\)  & .5800  & .7924
\\ \cline{1-5}

\end{tabular}
	\caption{The first Grigorchuk group: upper bounds computed by function \texttt{IsSubExp\_rec} without optimization.  The stars~\(*\) mark optimal and almost optimal data. The optimal weights are taken from Bartholdi~\cite{Bar98} and have been normalized to sum up to $1$.  Levels are defined at the end of~\cref{seq:Mealy}. }\label{table:resultGri}
\end{table}

\begin{table}[p]
\centering
	\begin{tabular}{|c|c|c|c|c|c|}
\cline{1-6}
level & update &target~\(\eta\)& obtained~\(\eta\)  & obtained weights & obtained \(\alpha\)
\\ \cline{1-6}

\multirow{1}{*}{1}  &
\(4\) &\(.90\)&~\(.8107^*\) &~\([ .3052, .3475, 
  .2243, .1236]^*
\) & \(.7676^*\) 
\\ \cline{1-6}
\multirow{2}{*}{2}  & \(4\)& \(.90\) & ~\(.8121\)  &~\([.3072,.3465,.1229,.2236]\) & \(.8063
\)
\\ \cline{2-6}& \(4\)& \(.72\) & ~\(.7166\)  &~\([.3069,.3466,.2068,.1399]\) & \(.9464\)
\\ \cline{1-6}
\multirow{2}{*}{3}  & \(4 \) or \(10\)& \(.90\) & ~\(.8889\)  &~\([.25,.25,.25,.25]\) & \(.9464
\)
\\ \cline{2-6}

 &\(4\) & \(.65\) & \(.6477\)  &  \( [ .4789, .2606, .2059,.0548]
\) & \(.8273\)
\\ \cline{2-6}

 &\(10\) & \(.63\) & \(.6287\)  &  \( [ .4004, .2999, .2345,.0654]
\) & \(.8176\)
\\ \cline{1-6}
\end{tabular}
	\caption{The first Grigorchuk group: upper bounds computed by function \texttt{IsSubExp\_opt} with optimization of  weights every update rounds, starting with uniform weights. The stars~\(*\) mark optimal and almost optimal data.}\label{table:resultGriopt}
\end{table}

\begin{table}[p]
\centering
	\begin{tabular}{|c|c|c|c|c|}

\cline{1-5}
Automaton &obtained~\(\eta\) & obtained \(\alpha\) & obtained weights & comments

\\ \cline{1-5}
\multirow{2}{*}{Fig.~\ref{fig:T1}}  & \multirow{2}{*}{\(.6450\)} & \multirow{2}{*}{\(.8034\)}  &\multirow{2}{*}{\([ .3352, .1899, .1899, .2849 ]\) }&  New automaton group
\\ 
&&&&weights obt. by optimization

\\ \cline{1-5}
\multirow{2}{*}{Fig.~\ref{fig:mNote}, from \cite{Bri17}} &
\(.8188\) &  \(.9123\) &\([1.,0.,0.]\)&prev. bound from~\cite{Bri17}

\\ \cline{2-4}
                        &\(.8300\) & \(.9178\) & \([.25,.25,.25,.25]\) &  was~\(.9396\)

 \\ \cline{1-5}
\multirow{2}{*}{Fig.~\ref{fig:Y}}  &\multirow{2}{*}{\(.8398\)} & \multirow{2}{*}{\(.9177\)} &\multirow{2}{*}{\([ .3785, .2655, .3561 ]\)}&  New automaton group

\\ 
 & & & &weights obt. by optimization

\\ \cline{1-5}

\end{tabular}
	\caption{Some new upper bounds computed.}\label{table:resultNew}
\end{table}

{\bf Previous works:}
As mentionned above, the underlying automaton allows the class of automata groups to be efficiently studied from an algorithmic point of view. For instance the word problem can be efficiently solved as well as the order problem~\cite{BBSZ13} or the Engel problem~\cite{Bar16} in some specific classes. Two main packages have been developed to this end: \texttt{AutomGrp}~\cite{AutomGrp} by Muntyan and Savchuk and \texttt{FR}~\cite{FR} by Bartholdi. Both are implemented in the \texttt{GAP} language~\cite{GAP}, specialized in computational algebra.

The situation is more contrasted for growth related question:
a first experimental approach was attempt by Reznykov and Sushchanskii~\cite{ReSu06}, where  the growth function up to a certain length  is computed, which allows the author to deduce conjectures from the observed values. Using this approach, they exhibit the smallest automaton semigroup with intermediate growth. Yet, no exact result can be inferred from this work and, as mentionned in~\cite{Res03}, the first terms of the growth function can be misleading. For instance, it follows from \cite{BaEr14, Bri14} that the growth may look exponential on large balls and still be intermediate at even larger scales. Notice that an efficient algorithm to compute the growth of automata groups, based on minimization, is provided by~\cite{KMP12}.

In \texttt{FR}, no function trying to determine the growth behaviour of an automaton group is provided (the functions computing the polynomial or exponential growth  degree are dealing with a different notion, that of activity introduced by Sidki~\cite{Sid00}).

In the package \texttt{AutomGrp}, a function testing if the automaton generates a group of subexponential growth is proposed. The tactic used is also based on  the contraction method and can be seen as the simplest version of  our algorithm: it looks if there exists a length~\(\ell\) such as every words of length~\(\ell\) are contracting. Our present approach refines this in several ways. 

First we no longer look for elements of the same length, but rather use the concept of essentially geodesic generating sets, introduced in Section~\ref{sec:eggs}. It permits to reduce drastically the number of elements whose contraction is required in order to assure subexponentiality. This allows our algorithm to detect groups of intermediate growth that are missed by~\texttt{AutomGrp}, such as the group generated by the automaton~\cref{fig:mNote}. 

Secondly, we use weights in the computation of contraction and actually implement an algorithm to numerically optimize them. Note that our dynamical procedures also provide an explicit bound when the procedure stops (but the bound is possibly trivial).  We stress that our code relies on some functions developed in \texttt{AutomGrp}, the most important being the rewriting system for groups.\\

{\bf Results obtained:}
Our approach produces several positive results:
\setlist[description]{font=\normalfont\itshape\textbullet\space}
\begin{description}
	\item[On the first Grigorchuk group:] our algorithm allows to retrieve numerically the optimal bound and weight system, see~\cref{table:resultGri,table:resultGriopt}.  Optimality of $\alpha=0.7674...$ is known from~\cite{ErZh18}. Throughout the paper, the numerical values we obtain are rounded up to the fourth digit.
	\item[Improvement of upper bounds:]  for the automata  in \cite{Bri17}, depicted in ~\cref{fig:mNote}, we improve the best known upper bound from $0.9396$ to $0.9123$, see~\cref{table:resultNew}.
	\item[New groups of intermediate growth:] we discovered several new automata generating group of intermediate growth (see \cref{fig:T1,fig:Y} and~\cref{table:resultNew}.). It should be noticed that these are not just twists of known automata but rather seem to belong to new families.
\end{description}

{\bf Perspectives:}
Our approach can be naturally extended to (homogeneous) spinal groups~\cite{BaSu01} which generalize Grigorchuk groups and are known to be of intermediate growth. Moreover, the new groups of intermediate growth described in~\cref{fig:T1,fig:mNote,fig:Y} seem to belong to large families of groups with intermediate growth. A systematic study of these families is in order. One aim of the present paper was to provide tools for this study.

From a computational point of view, we can aim for two main improvements: firstly, the procedures could be implemented in an other, more efficient, language. Secondly, the procedure of exploration is tractable to parallelization, which would greatly improve the performances of our methods. 

Finally, we stress that many algorithmic problems are undecidable for automata groups, and that it is unknown if there exists an algorithm proving that an automaton is contracting. On the other hand, some decidable classes of automata, e.g. the subclass of bounded automata~\cite[Theorem~6.5]{Nek05} which contains most known automata groups of intermediate growth, force the generated group to be contracting in a sense. It would be interesting to explore these classes with our algorithms. Moreover, one can wonder if the growth rate is decidable in the class of bounded automata, which contains groups of all types of growths.\\

{\bf Organization of the paper.} Basic facts about Mealy automata and their associated groups are briefly presented in \cref{sec:aut}, together with the examples under focus in the paper. Word metric and growth rates are defined in \cref{sec:growth}. The crucial notion of eggs (essentially geodesic generating sets) is defined and studied in \cref{sec:eggs}. The strong contraction method for subexponential growth is described in \cref{sec:contcrit}. \cref{sec:algo} is devoted to the description of  (semi)-algorithms that provide upper bounds. Experimental data obtained with them and comments appear in \cref{sec:data}. The last \cref{sec:superpoly} gives a criterion ensuring that the groups we consider do not have polynomial growth. \cref{app:fig} gathers the pictures of diagrams and Schreier graphs of all automata groups studied here. The \texttt{GAP} code is available at \texttt{https://www.irif.fr/\textasciitilde godin/automatongrowth.html}.\\

{\bf Acknowledgements.} We wish to thank Laurent Bartholdi and Dyma Savchuk for useful comments about \texttt{GAP} implementation.

J.B. and Th.G. were partially supported by ANR-16-CE40-0022-01 AGIRA. Th.G. was partially supported by the Academy of Finland grant 296018.

\section{Automata groups}\label{sec:aut}


We briefly review definitions and elementary facts about Mealy automata and self-similar group. The reader is refered to~\cite{Nek05,Zuk06} for comprehensive description.

\subsection{Mealy automata}\label{seq:Mealy}
A \emph{Mealy automaton} is a complete deterministic letter-to-letter transducer \(\auta = \auttuple\), where \(\QQ\) and \(\XX\) are finite sets respectively called the the \emph{stateset} and the \emph{alphabet}, and $\delta= (\delta_i : Q \to \QQ)_{i \in \XX}$, $\rho = (\rho_q : \XX \to \XX )_{q \in \QQ}$ are respectively called the \emph{transition} and \emph{production} functions. Examples of such Mealy automata are depicted Figure~\ref{fig:Mealy}. The transition and production functions can be extended to words as follows:  see~\(\aut{A}\) as an automaton with  input and  output tapes, thus
defining mappings from input words over~$\Sigma$ to output words
over~$\XX$.
Formally, for~\(q\in \QQ\), the map~$\rho_q\colon\XX^* \rightarrow \XX^*$,
extending~$\rho_q\colon\XX \rightarrow \XX$, is defined recursively by:
\begin{equation}\label{eq:acttree}
\forall i \in \XX, \ \forall \mot{s} \in \XX^*, \qquad
\rho_q(i\mot{s}) = \rho_q(i)\rho_{\delta_i(q)}(\mot{s}) \:.
\end{equation}

Observe that $\rho_q$ preserves the length of words in $\XX^*$. 
We can also extend the map $\rho$ to words of states $\mot{u} \in \QQ^*$ by composing the production functions associated with the letters of $\mot{u}$:
\begin{equation}
 \forall q \in \QQ, \ \forall \mot{u} \in \QQ^*, \qquad
\rho_{q\mot{u}} = \rho_{\mot{u}}\circ \rho_{q}\:.
\end{equation}
Therefore the production functions $\rho_q : \XX^* \to \XX^* $ of an automaton $\auta$ generate a semigroup $\presp{\auta}:= \{\rho_{\mot{u}} : \XX^* \to \XX^* | \mot{u} \in \QQ^*\}$, all of which elements preserve the length of words in $\XX^*$. 

A Mealy automaton is said to be \emph{invertible} whenever $\rho_q$ is a permutation of the alphabet for every $q \in \QQ$. In this case, all functions $\rho_q : \XX^* \to \XX^*$ are invertible and the automaton actually generates a group:
\[
\pres{\auta}:=\pres{\rho_q^{\pm 1}:\XX^*\to\XX^*|q \in \QQ}=\left\{\rho_{\mot{u}}:\XX^*\to\XX^*|\mot{u} \in \left( \QQ \cup \QQ^{-1}\right)^* \right\}\:.
\] 
It is not difficult to find another (symmetrized) Mealy automaton generating this group as a semigroup. Its stateset can be taken as $Q \cup Q^{-1}$ such that $\rho_q^{-1}=\rho_{q^{-1}}$. A group of the form $\pres{\auta}$ for an invertible automaton is called an \emph{automaton group}.
Such a group acts naturally on the set~\(\XX^*\).

Observe that given an integer $k$ and a Mealy automaton~$\auta=\auttuple$ with alphabet~$\XX$, one can associate to it a Mealy automaton~$(\QQ,\XX^k,\delta,\rho)$ where we identify the function~$\rho_q$ acting on~$\XX$ with the same function acting on~$\XX^k$ via (\ref{eq:acttree}). This just amounts to replacing the alphabet by the set of syllables of length~$k$. With a slight abuse of language, we call the latter automaton the level~$k$ of~$\auta$. Notice that this level~$k$ automaton generates the same group as the original automaton.

\subsection{Self-similar groups}
One can describe the automaton group through this action via   \emph{wreath recursion} and obtain a so-called \emph{self-similar} group. First let us recall a few definitions.
Let us write \(\Gamma \curvearrowright X\) the action of a group \(\Gamma\) on a set \(X\) and $\act{\gamma}{x} \in X$
the transformation of \(x \in X\) induced by \(\gamma \in \Gamma\).
The \emph{permutational wreath product} \(G \wr_\XX \Sym(\XX)\) of  a group \(G\) over a set $\XX$ is the set $G^\XX \times \Sym(\XX)$ together with the operation 
\[
\wrpr{\sect{g}{x_1},\hdots, \sect{g}{x_d}}{\perm{g}}.\wrpr{\sect{h}{x_1},\hdots, \sect{h}{x_d}}{\perm{h}}=\wrpr{\sect{g}{x_1}\sect{h}{\perm{g}^{-1}.x_1},\hdots, \sect{g}{x_d}\sect{h}{\perm{g}^{-1}.x_d}}{\perm{g}\perm{h}}
\] 
for the obvious action of $\Sym(\XX)$ on $\XX  = \lbrace x_1, x_2, \hdots, x_d \rbrace$.

Assume we are given a group $G$ together with an injective homomorphism $\psi:G \hookrightarrow G \wr_\XX \Sym(\XX)$. Then one can describe the group elements with  a recursive formula, where we canonically identify $g$ and $\psi(g)$
\begin{equation}\label{eq:sec}
g =\psi(g)= \wrpr{\sect{g}{x_1},\hdots, \sect{g}{x_d}}{\perm{g}} \:.
\end{equation}
The group element  $\sect{g}{x_i}$ is called the \emph{section} of $g$ in $x_i$ and $\perm{g}$ is the \emph{action} of $g$. One can extend the section to words : $\forall \mot{u}x \in \XX^*, \sect{g}{\mot{u}x}  = \sect{{\sect{g}{\mot{u}}}}{x}$.

One can also extend the action of $G$ to $\XX^*$ by $\forall x\mot{u} \in \XX^*, g.x\mot{u}=(\pi(g).x)\sect{g}{x}.\mot{u}$.
The action of \(G\) on \(\XX^*\), is called \emph{self-similar} if 
\[\forall g \in G, \forall \mot{u} \in \XX^*, \forall x \in \XX, \exists h \in G, \exists y \in \XX, \act{g}{x\mot{u}} = y\act{h}{\mot{u}}\:. \]

If, for every element $g$ of $G$ the set $\lbrace \sect{g}{\mot{u}}, \mot{u} \in \SS^* \rbrace$ is finite, the action is said to be  \emph{finite state}. It is easy to see that such a finite state self-similar action corresponds to exactly one automaton group, possibly with infinite stateset. A quick way to describe an automaton group is to give the section decomposition (\ref{eq:sec}) of all elements of the stateset.

The correspondence between Mealy automata and self-similar groups can be summarized as: 
\begin{equation*}
\rho_q(x)=y \textrm{ and }\delta_x(q)=p \quad \iff \quad  \sect{q}{x} = p \text{ and } \perm{q}.x= y\:,
\end{equation*} 
which we represent  by the arrow notation $q \xrightarrow{x\mid y} p \ \in \auta$. The stateset of the automaton corresponds to a generating set of the self-similar group.

\subsection{Diagram, dual automaton and Schreier graph}

The arrow representation permits to describe a Mealy automata by its \emph{diagram}, which is a labelled oriented graph. The vertices of the diagram are labelled by the states. From each vertex $q$ and letter $x$, there is an oriented edge from $q$ to $p=\delta_x(q)$ which we label by  $x \mid y$ where $y=\rho_q(x)$. The diagrams of most groups studied in this paper are pictured in \cref{app:fig}.

The roles of the stateset and the alphabet in an automaton are symmetric in the definition. The \emph{dual automaton} of $\auta=\auttuple$ is the automata $\hat{\auta}=(\XX,\QQ,\rho,\delta)$ obtained by exchanging them. This amounts to
\[
q \xrightarrow{x\mid y} p  \ \in \auta \quad \iff \quad x \xrightarrow{q\mid p} y  \in \ \hat{\auta} \quad \textrm{ for all }p,q,x,y.
\]

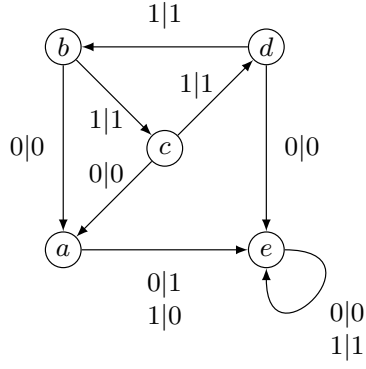
\begin{figure}[h!]

\begin{center}

{
\scalebox{1.125}{
\footnotesize{
	\begin{tikzpicture}[->,>=latex,node distance=12mm]

		\node[state] (c)  {$c$};
				\node[state] (a) [left of=c,below of=c] {$a$};
				\node[state] (b) [left of=c, above of=c] {$b$};
				\node[state] (d) [right of=c, above of=c] {$d$};
				\node[state] (e) [right of=c,below of=c] {$e$};
			\path
		
				(a)	edge	node[below=.1cm]{\(\begin{array}{c} {0|1} \\ {1|0} \end{array}\)}	(e)
				(b) edge 	node[left=.1cm]{\(0|0\)} (a)
							(b) edge 	node[below=.1cm, pos=0.4]{\(1|1\)} (c)
				(c) edge 	node[above=.1cm, pos=0.6]{\(0|0\)} (a)
				(c) edge 	node[above=.1cm, pos=0.25]{\(1|1\)} (d)
				(d) edge 	node[right=.1cm]{\(0|0\)} (e)
				(d) edge 	node[above=.1cm]{\(1|1\)} (b)
				(e) edge[min distance=10mm,out=00,in=270,looseness=10] node[below right =-1mm and -1mm]{\(\begin{array}{c} 0|0\\ 1|1 \end{array}\)}	(e)
		
		;
	\end{tikzpicture}
	}
	}	
}
\end{center}
\caption{The automata generating the first Grigorchuk group.}\label{fig:Mealy}
\end{figure}

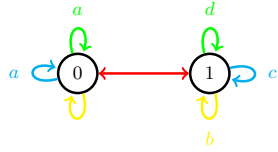
\begin{figure}[h!]

\begin{center}
{
\scalebox{.8}{
\footnotesize{
	\begin{tikzpicture}[very thick]
			\node[state, inner sep =4pt] (0)  {$0$};
				\node[state, inner sep =4pt] (1) [right = 1.5 cm of 0] {$1$};
					
			\path
				(0)	edge[red,<->]	node[above=.1cm]{\(\)}	(1)			
				(0)	edge[yellow,loop below,->]	node[below=.1cm]{\(\)} (0)	
				(1)	edge[yellow,loop below,->]	node[below=.1cm]{\(b\)} (1)				

					
				(0)	edge[cyan,loop left,->]	node[left=.1cm]{\(a\)} (0)	
				(1)	edge[cyan,loop right,->]	node[right=.1cm]{\(c\)} (1)	

				(0)	edge[green,loop above,->]	node[above=.1cm]{\(a\)} (0)	
				(1)	edge[green,loop above,->]	node[above=.1cm]{\(d\)}(1)

		;
	\end{tikzpicture}
	}

	}	
}

\end{center}
\caption{The  Schreier graph on level 1 of the automata generating the first Grigorchuk group.}\label{fig:Gri3}
\end{figure}

When we draw the diagram of the dual automaton $\hat{\auta}$, we obtain the \emph{Schreier graph} of the action of $\langle \auta \rangle$ on the alphabet $\SS$, i.e. the graph with vertex set $\SS$ and edges from $x$ to $y=q.x=\rho_q(x)$ labelled by $q$.

\subsection{Examples}\label{sec:examples}

In this paper we will focus on a few iconic examples.

\begin{itemize}
\item The first Grigorchuk group is the most famous group of intermediate growth. A stateset is given by 
\[
a=\langle e,e \rangle (0,1), \quad b=\langle a,c \rangle, \quad c=\langle a,d \rangle, \quad d=\langle e,b \rangle.
\]
Its diagram and Schreier graphs are depicted in \cref{fig:Mealy}. We also apply our algorithm to disguised versions of the first Grigorchuk, where we look at it as generated by the automata acting on the second or third levels of the tree. They provide interesting situations where we know the actual optimal bound but face a more challenging computation challenge. The diagram and Schreier graph of the automaton acting on the third level are drawn in \cref{fig:Gri3}. These Schreier graphs have been extensively studied by Vorobets~\cite{Vor10}.

\item We introduce a close relative of the Grigorchuk groups. It acts on an alphabet with $6$ letters. Its stateset is given by
\begin{align*}
a &=\langle e,e,e,e,e,e \rangle (1,2)(3,4)(5,6) \\
b&= \langle e,e,e,e,e,b \rangle (4,5) \\
c&= \langle a,e,e,e,e,c \rangle (2,3) \\
d&= \langle a,e,e,e,e,d \rangle (2,3)(4,5)
\end{align*}
and its diagram and Schreier graph depicted in \cref{fig:T1}.
Its introduction is motivated by the similarities with the Schreier graph descriptions the Grigorchuk group in~\cite{Vor10} (compare \cref{fig:Gri3}) and the fragmentations of dihedral actions as introduced by Nekrashevych in~\cite{Nek18}. Note that the elementary relations $a^2=b^2=c^2=d^2=bcd=e$ of the Grigorchuk groups are satisfied, however this group does not belong to the class of Grigorchuk groups introduced in \cite{Gri85} because it contains elements of order $6$.

\item The group introduced by the first author in~\cite{Bri17} and whose diagram and Schreier graphs are depicted in \cref{fig:mNote} was a motivation for the present work. It is related to Wilson's groups of non-uniform growth~\cite{Wil04,Wilson2003,Brieussel2009}. It is generated by an involution and an element of order $3$.

\item Similarly to the previous example, we introduce a new automaton group which is a close relative of a group of intermediate growth constructed by Bartholdi in relationship with non-uniform growth~\cite{BartholdiNonU2003}. Its diagram and Schreier graphs are depicted in \cref{fig:Y}. It is generated by three involutions and has exponential activity in the sense of Sidki.

\item Finally, we provide an example of an automaton group with $9$ states acting on an alphabet with $17$ letters. Its Schreier graph has the shape of an $X$ but the labellings are inspired by the Schreier graph of the first Grigorchuk group. Its stateset is:
\begin{align*}
a &=\langle e,e,e,e,e,e,e,e,e,e,e,e,e,e,e,e,e\rangle (1,2)(6,7)(12,13)(15,16)\\
b &=\langle e,e,e,e,e,e,e,e,e,e,e,e,e,e,e,e,b\rangle (4,5)(7,8)(1,10)(14,15)\\
c &=\langle e,e,e,e,e,e,e,e,a,e,e,e,e,e,e,e,c\rangle (4,5)(14,15)\\
d &=\langle e,e,e,e,e,e,e,e,a,e,e,e,e,e,e,e,d\rangle (7,8)(1,10)\\
a' &=\langle e,e,e,e,e,e,e,e,e,e,e,e,e,e,e,e,b\rangle (3,4)(8,9)(10,11)(1,14)\\
b' &=\langle e,e,e,e,b',e,e,e,e,e,e,e,e,e,e,e,e\rangle (2,3)(1,6)(11,12)(16,17)\\
c' &=\langle e,e,e,e,c',e,e,e,e,e,e,e,a',e,e,e,e\rangle (2,3)(16,17)\\
d' &=\langle e,e,e,e,d',e,e,e,e,e,e,e,a',e,e,e,e\rangle (1,6)(11,12).\\
\end{align*}
It satisfies the relations $a^2=b^2=c^2=d^2=bcd=a'^2=b'^2=c'^2=d'^2=b'c'd'=e$.
\end{itemize}

\section{Word norms and growth}\label{sec:growth}


\subsection{Word lengths and word norms}

Let \(G\) be a group. A subset \(S\) is \emph{generating} (in the sense of semigroups) if for any \(g \in G\) there exists \(s_1,\dots,s_n \in S\) such that \(g=s_1\dots s_n\). In other terms, there is a word in the free semi-group $S^*$ generated by $S$ that equals $g$ once evaluated in $G$. By convention, we assume that the neutral element $e$ is not in $S$.
A \emph{weight} function on $S$ is a positive function $\pi:S\to \R_{>0}$. A subset $S$ of $G$ is \emph{symmetric} if $S=S^{-1}$ and in this case $\pi$ is symmetric if $\pi(s^{-1})=\pi(s)$ for all $s$ in $S$. Throughout the paper, we consider triples $(G,S,\pi)$, where $G$ is a group together with a finite generating set $S$ and a weight $\pi$ on it. 

For any $w=s_1\dots s_n$ in $S^*$, we call \emph{length} of $w$ with respect to $\pi$ the number
\[
|w|_\pi:=\sum_{i=1}^n \pi(s_i).
\]
In particular $|s|_{\pi}=\pi(s)$ for all $s \in S$. By convention the empty sum is zero. 
For any  $g$ in \(G\), we call \emph{norm} of $G$ with respect to $(S,\pi)$ the number
\[
\|g\|_{S,\pi}:=\inf\left\{\left|w\right|_\pi : w \in S, w=_Gg\right\}.
\]
When $\pi$ is constant equal to one, we recover the usual word norm with respect to $S$. In this case, we simply write $\|\cdot\|_S$. The terminology is justified by the

\begin{proposition}\label{wnorm}
The function $\|\cdot\|_{S,\pi}:G\to \R_{\ge 0}$  satisfies 
\begin{enumerate}[label=(\alph*)]
\item \(\forall g,h \in G, \|gh\|_{S,\pi}\le \|g\|_{S,\pi}+\|h\|_{S,\pi}\),
\item \(\|g\|_{S,\pi}=0\) if and only if \(g=e\) is the neutral element.
\end{enumerate}
Moreover, the infimum is in fact a minimum.
\end{proposition}

A word $w$ is a \emph{minimal representative} of a group element $g$ if $|w|_\pi=\|g\|_{S,\pi}$. It always exists but may not be unique. Observe that a subword of a minimal word is also minimal by (a). Regarding generators, we always have $|s|_\pi \ge \|s\|_{S,\pi}$, but the converse inequality may not be true. However when we have such a strict inequality, the generator $s$ never appears in minimal representative words, so the generating set is in fact redundant.

\begin{proof}
Let $g \neq e$. By assumption there is a word $w$ in $S$ representing $g$, so $\|g\|_{S,\pi}$ is finite. Since $\pi_0=\min\{\pi(s)|s\in S\}>0$, there are only finitely many words in $S^*$ of $\pi$-length less than $|w|_\pi$, so the infimum is a minimum. Moreover only the empty word (representing the neutral element) has zero $\pi$-length. This prove (b). Finally let $w_g$ and $w_h$ be minimal representative words of $g$ and $h$, then $w_gw_h$ represents $gh$, assertion (a) follows.
\end{proof}

This word norm induces a metric \(d_{S,\pi}(x,y):=\|x^{-1}y\|_{S,\pi}\) on \(G\). Recall that a metric is a function $d:G\times G \to \R_{\ge 0}$ such that for all $x,y,z \in G$, we have $d(x,z) \le d(x,y)+d(y,z)$ and that $d(x,y)=0$ if and only if $x=y$. This metric is left-invariant under the group action in the sense that $d_{S,\pi}(gx,gy)=d_{S,\pi}(x,y)$ for all $g,x,y$ in $G$. Observe that this metric is symmetric, which means $d_{S,\pi}(x,y)=d_{s,\pi}(y,x)$ for all $x,y$, if and only if both $S$ and $\pi$ are symmetric.

When we identify the group $G$ with the vertices of its Cayley graph \(\Gamma_{G,S}\), the metric $d_{S,\pi}$ coincides with the metric inherited by the vertex set when we declare that each oriented edge labelled by $s$ has length $\pi(s)$. It implies that the path in \(\Gamma_{G,S}\) defined by following edges labelled by the letters of a minimal word \(w\) is a geodesic (i.e. shortest) path from the starting point $x$ to the end point $y=xw$.

\subsection{Growth functions and growth rate}

\begin{definition} The \emph{ball} of radius $\ell$ and center $x$ in $(G,d_{S,\pi})$ is the set
\[
B_{G,S,\pi}(x,\ell):=\{y \in G:d_{S,\pi}(x,y)\le \ell\}.
\]
The \emph{growth} function of \(G\) with respect to \((S,\pi)\) counts the sizes of the balls 
\[
\gamma_{G,S,\pi}(\ell)=\#B_{G,S,\pi}(x,\ell)\:.
\]
By left-invariance, the size of a ball in a group depends only on the radius $\ell$ and not on the center~$x$. To fix ideas, one may choose \(x=e\).
\end{definition}

Let us say that two functions \(f,g\) are \emph{equivalent} when there exists \(c>0\) such that \(g(\frac{1}{c}\ell) \le f(\ell) \le g(c\ell)\) for all \(\ell\ge 0\). The next lemma ensures that the equivalence class of $\grgrow_{G,S,\pi}(\ell)$ does not depend on \(S\) nor on \(\pi\). This class is called the growth \emph{rate} of \(G\), and denoted by abuse of notation~$\grgrow_{G}(\ell)$.

\begin{lemma}\label{lem:growthindegen}
Let \((S,\pi)\) and \((T,\rho)\) be two finite weighted generating sets of the group \(G\), then there exists $c>0$ such that
$\grgrow_{G,S,\pi}(\ell)\le \grgrow_{G,T,\rho}(c\ell)$ for all $\ell$.
\end{lemma}

\begin{proof}Let  $c:=\max\left\{\right.\frac{\|s\|_{T,\rho}}{\|s\|_{S,\pi}}\left\vert s\in S\right\}\in (0,\infty)$. If $g \in B_{G,S,\pi}(\ell)$ then there exists a minimal representative word $w=s_1\dots s_n$ in $S^*$ so that
\[
\ell \ge \|g\|_{S,\pi}=|w|_\pi=\sum_{i=1}^n\pi(s_i)=\sum_{i=1}^n|s_i|_\pi=\sum_{i=1}^n\|s_i\|_{S,\pi}.
\]
A fortiori 
\[
\|g\|_{T,\rho} \le \sum_{i=1}^n \|s_i\|_{T,\rho} \le c\sum_{i=1}^n\|s_i\|_{S,\pi} \le c\ell.
\]
\end{proof}

\begin{definition} A group $G$ has
\begin{itemize}
\item \emph{polynomial} growth if $\grgrow_G(\ell)$ is bounded above by a polynomial function,
\item \emph{exponential} growth if there is $c>0$ such that $\grgrow_G(\ell)\ge \exp(c\ell)$ for all $\ell$,
\item  \emph{intermediate} growth if  $\grgrow_G(\ell)$ is  greater than any polynomial function and lesser than any exponential function:
\[
\forall d \in \Z, \alpha \in \R_{>0}, \exists N, \forall \ell \geq N, \ell^d \leq \grgrow_{G}(\ell) \leq \exp{(\alpha \ell)}.
\]
\end{itemize}
\end{definition}

\section{Essentially geodesic generating sets (eggs)}\label{sec:eggs}


This section is devoted to the key notion of essentially geodesic generating sets, in short eggs, of a triple $(G,S,\pi)$.

\subsection{Definition and examples}

\begin{definition}\label{defegg}
A finite subset \(\Delta\) of \(S^*\) is an {\emph{essentially geodesic generating}} set (in short, an egg) for $(G,S,\pi)$ if there is a finite subset \(F\) of  \(S^*\) such that for any \(g \in G\), there exist \( \d_1,\dots,\d_n \in \Delta\) and \( h \in F\) such that the word \(\d_1\dots\d_n h\) is a minimal representative of \(g\) with respect to $(S,\pi)$.
\end{definition}

Clearly, the words \(\d_i\) of the definition have to be minimal themselves. For example, \(S\) itself is an egg, with $F$ the empty set. Let us give more examples. 

Let \(\O\subset G\). Define its boundary as \(\partial \O:=\{g \in \O:\exists s \in S, gs \notin \O\}\). Denote \(\Int(\O):=\O \setminus \partial \O\) the interior of \(\O\). In the Cayley graph, vertices in the boundary of \(\O\) are precisely those with at least one neighbor outside \(\O\).

\begin{lemma}\label{lem-bord}
Assume \(\O\subset G\) is finite and \(e\in \Int(\O)\). Then a set \(\Delta\) of minimal representative words of  the boundary \(\partial \O\) is an egg of $G$.
\end{lemma}

\begin{proof}
Given \(g \in G\), let \(w=s_1\dots s_n\) be a minimal representative word for \((S,\pi)\). Starting from \(e\) and following the edges labelled by the letters of \(w\) provides a geodesic path to \(g\) in the Cayley graph. If \(g \notin \Int(\O)\), this path must cross \(\partial \O\). This means there is a non-trivial prefix (necessarily minimal) \(w_1\) of \(w\) representing an element in \(\partial \O\). By assumption, there is \(\delta_1\in\Delta\) of the same length as \(w_1\) such that \(\d_1=_G w_1\). Then \(\delta_1^{-1}g\) is an element in \(G\) at distance from \(e\) less than \(d_{S,\pi}(e,g)-\pi_0\), where $\pi_0:=\min\{\pi(s)|s\in S\}>0$. Arguing by induction, we can write \(g=\d_1\dots \d_k h\) with \(\d_i \in \Delta\) and \(h \in \Int(\O)=:F\) which is finite.
\end{proof}

Of course not all eggs are of this form, because any set of words containing an egg is still an egg. But there are also eggs containing no such \(\partial \O\) as seen by the following.

\begin{example}\label{egg-no-bord}
Let \(S=\lbrace a^{\pm1},b^{\pm1}\rbrace\) be the usual generating set of \(\Z^2\) where \(a=(1,0),b=(0,1)\). Then for integers \(p,q>0\), the set \(\D=\lbrace a^{\pm p},b^{\pm q}\rbrace\) is an egg, because any $g \in \Z^2$ can be written minimally $g=a^{\lambda p}b^{\mu q}a^{r}b^s$ with $\lambda,\mu,r,s$ integers and $|r|<p$, $|s|<q$. However the minimal word \((ab)^k\) has no prefix in \(\D\).
\end{example}

\subsection{Stability under quotients}\label{sec:stabquo}

Let $\til{G}$ be a group together with a quotient group $G$. Denote by \(q:\til{G}\to G\) the quotient map. We assume that $\til{S}$ is a finite generating set of $\til{G}$ and that $q|_{\til{S}}$ is injective, inducing a bijection between $\til{S}$ and $S:=q(\til{S})$. Given a word $\til{w}=\til{s_1}\dots \til{s_n}$ in  $\til{S}$, we systematically write $w=s_1\dots s_n$ instead of $q(\til{w})=q(\til{s_1})\dots q(\til{s_n})$ the corresponding word in $S$. Clearly, $S$ is a generating set of $G$.

Let $\pi:\til{S}\to \R_{>0}$ be a weight function, it naturally induces a weight function on $S$ still denoted $\pi$. For any $\til{g}$ in $\til{G}$ we observe that $\|q(\til{g})\|_{S,\pi} \le \|\til{g}\|_{\til{S},\pi}$. Indeed, any word in $\til{S}$ representing $\til{g}$ corresponds to a word in $S$ of the same $\pi$-length representing $q(\til{g})$, so the infimum defining the left-hand side is taken over a bigger set of words.

\begin{lemma}\label{lem-quotient}
In the setting above, assume that  \(\til{\Delta}\) is an egg for \((\til{G},\til{S},\pi)\). Then  \(\Delta:=q(\til{\Delta})\) is an egg for \((G,S,\pi)\).
\end{lemma}

\begin{proof}
Let $g$ belong to $G$ and $w$ be a minimal representative word of $g$ in $S$. Then $\til{w}$ is a minimal word in $\til{S}$ (by the observation before the lemma). As $\til{\Delta}$ is an egg, we have $\til{w}=_G \til{\delta_1}\dots \til{\delta_n}\til{h}=:\til{w_1}$, a minimal word with $\til{\delta_i}$ in $\til{\Delta}$ and $\til{h}$ in some finite set $\til{F}$ of words in $\til{S}$.  A fortiori, $w_1=\delta_1\dots\delta_n h$ is the required representative word of $g$.
\end{proof}

\begin{example}
For \(G\) the first Grigorchuk group as depicted in \cref{fig:Mealy}, it is natural to consider \(\til{G}=\langle a,b,c,d|a^2=b^2=c^2=d^2=bcd=e\rangle\), which is isomorphic to the free product \(\Z/2\Z \ast \left(\Z/2\Z\oplus\Z/2\Z\right)\).
\end{example}

\subsection{Stability of eggs under change of weights}\label{sec:stabweight}

Consider an egg $\D$ for $(G,S,\pi)$. We want to know under what condition the set $\D$ is still an egg for $(G,S,\pi')$. Clearly this is the case if among words in $S$, minimality for $\|\cdot\|_{S,\pi}$ is equivalent to minimality for $\|\cdot\|_{S,\pi'}$. 

\begin{lemma}
Let $G:=\F(X)$ be the free group with finite free generating set $X$. Let $S=X \cup X^{-1}$ and $\pi:S\to \R_{>0}$ be a weight. A word representative is minimal for $\|\cdot\|_{S,\pi}$ if and only if it is minimal for the uniform word norm $\|\cdot\|_S$.

In particular, a set $\Delta$ is an egg for $\|\cdot\|_{S,\pi}$ if and only if it is an egg for the uniform word norm~$\|\cdot\|_S$.
\end{lemma}

\begin{proof}
As the Cayley graph is a tree, minimal representatives are unique and correspond to non-backtracking paths. Therefore they coincide for both norms.
\end{proof}

We can obtain a similar stability result in the context of free products of finite groups, but the weights need to satisfy some condition.

\begin{lemma}\label{minfred}
Let $G=\ast_{i\in I} G_i$ be a free product of finitely many finite groups. Let $S=\cup_{i\in I} G_i\setminus\{e\}$ and $\pi:S\to \R_{> 0}$ be a weight function such that
\begin{eqnarray}\label{cond}
\forall i \in I, \forall s,s' \in G_i\setminus\{e\}, \quad \pi(ss') < \pi(s)+\pi(s').
\end{eqnarray}
Then the following are equivalent:
\begin{enumerate}[label=(\roman*)]
\item a word $w=s_1\dots s_n$ is minimal for $\|\cdot\|_{S,\pi}$, 
\item the word $w$ is freely reduced, i.e. for all $1 \le k \le n-1$, the generators $s_k$ and $s_{k+1}$ belong to different factor groups $G_i$.
\end{enumerate}
\end{lemma}

A weight function satisfying condition (\ref{cond}) is called \emph{triangular}. This condition is essentially necessary because if there are $s,s' \in G_i$ with $s_0=ss'$ and $\pi(s_0)>\pi(s)+\pi(s')$ then $s_0$ would be a freely reduced word not minimal for~$\|\cdot\|_{S,\pi}$.

As the usual word norm satisfies condition (\ref{cond}), minimality for $\|\cdot\|_{S}$ and for $\|\cdot\|_{\pi,S}$ are equivalent. We deduce the:

\begin{corollary}\label{cor:eggweight}
Let $G=\ast_{i\in I} G_i$ and $S, \pi$ be as in Lemma \ref{minfred}. A set $\Delta$ is an egg for $\|\cdot\|_{S,\pi}$ if and only if it is an egg for the usual word norm.
\end{corollary}

\begin{proof}[Proof of Lemma \ref{minfred}]
If a word is minimal, then successive letters must belong to different factors, otherwise using (\ref{cond}) we would replace the successive letters $s_k$ and $s_{k+1}$ by their product, thus obtaining the same group element but with a representative of shorter $\pi$-length. 

Conversely assume that $w$ is freely reduced, but not minimal. Then there exists $w'=s_1'\dots s_m'=_Gw$ a minimal representative word for $\|\cdot\|_{S,\pi}$. By the first part of the proof, we can assume $w'$ is freely reduced. Now let $k_0:=\min\{k:s_k'\neq s_k\}$, then $s_{k_0}\dots s_n=_Gs_{k_0}'\dots s_m'$ 
which implies $s_m'^{-1}\dots s_{k_0}'^{-1}s_{k_0}\dots s_n=_G e$. We obtain a freely reduced representative word of the identity. As $G$ is a free product, this is a contradiction unless $w$ and $w'$ are the same word.
\end{proof}

The context of free products of finite groups is restrictive. However, it seems delicate to obtain more general results about stability of eggs, as shown by the following:

\begin{example}
Consider the group $\Z^2$ with $S=\{(\pm1,0),(0,\pm1),\pm(1,1)\}$ and two different weigths $\pi_1(\pm1,0)=\pi_1(0,\pm1)=1$, $\pi_1(\pm(1,1))=2$ and $\pi_2(\pm1,0)=\pi_2(0,\pm1)=1$, $\pi_2(\pm(1,1))=\l \in (0,2)$. Then the norm $\|\cdot\|_{S,\pi_1}$ coincides with the usual word metric on the grid so that $\Delta:=\{(\pm1,0),(0,\pm1)\}$ is an egg. However $\Delta$ is no longer an egg for $\|\cdot\|_{S,\pi_2}$ because $\|(n,n)\|_{S,\pi_2}=\l n$ is much smaller than the distance between $e$ and any fixed neighborhood of $(n,n)$ in the usual word norm which is at least~$2n-c$.
\end{example}

\section{The strong contraction method for subexponential growth}\label{sec:contcrit}


The following theorem will be our key tool to obtain upper bounds on growth exponents. The general strategy is due to Grigorchuk \cite{Gri83,Gri85}, the use of weights appeared in \cite{Bar98} and the following precise statement is taken from \cite{MuPa01}. This method is so far the only known method to obtain non-trivial upper bounds on growth exponents.

\begin{theorem}[Muchnik-Pak \cite{MuPa01}]\label{MP}
Given a triple $(G,S,\pi)$, assume that there exist \(\eta \in [0,1]\) and \(c_0>0\) together with embeddings for all $\ell$
\[
B_{G,S,\pi}(\ell) \hookrightarrow \bigcup_{\ell_1+\dots+\ell_d\le \eta \ell+c_0} K \times \prod_{i=1}^d B_{G,S,\pi}(\ell_i)
\]
where \(K\) is a fixed finite set. 
Then there is a constant \(c>0\) such that \(\gamma_{G,S,\pi}(\ell) \le e^{c\ell^\a}\) with \(\a=\frac{\log(d)}{\log(d)-\log(\eta)}\).
\end{theorem}

Note that the assumptions need to be satisfied by one weighted generating set and the conclusion then applies to the growth rate of \(G\), hence for all weighted generating set. However in practice, the constants \(\eta,  c_0,c\) and the finite set \(K\) depend on \((S,\pi)\). The key point in the sequel will be to find good generating sets and weights to lower $\eta$ and thus $\alpha$. Of course, this theorem is relevant only in the case $\eta<1$. To apply it, we will systematically use the following:

\begin{proposition}\label{prop-contract}
Let \(G\) be a self-similar group with \(\psi:G \hookrightarrow G \wr_{\XX} \Sym(\XX)\). Assume there is a weighted norm $\|\cdot\|_{S,\pi}$ with an egg~\(\D\) and a constant \(\eta\in [0,1]\) such that
\begin{equation}
  \tag{$\star$}
\forall \d \in \Delta,\quad \sum_{x \in \XX} \|\d_x\|_{S,\pi} \le \eta \|\d\|_{S,\pi}
  \label{eqn:contract}
\end{equation}
where \(\psi(\d)=\wrpr{\d_{x_1},\hdots,\d_{x_d}}{\s}\) is the self-similar image of \(\d\). Then the assumptions of Theorem \ref{MP} are satisfied with the same $\eta$.
\end{proposition}

We call the minimal such \(\eta=\eta_{S,\pi}(\D)\) the {\it contraction coefficient} of the egg \(\D\). In order to get upper bounds on growth, we will construct eggs with small contraction coefficient.

\begin{proof}
Take $K=\Sym(\XX)$. The permutational wreath products defines an embedding $G \hookrightarrow \Sym(\XX) \times G^\XX$ by declaring that the image of $g=\wrpr{g_{x_1},\hdots,g_{x_d}}{ \s}$ is the $d+1$-tuple $(\s,g_{x_1},\dots,g_{x_d})$. It is sufficient to prove that the image of each element $g$ in $B_{G,S,\pi}(\ell)$ is in the desired union, i.e. that $\sum_{x\in \XX} \|g_x\|_{S,\pi} \leq \eta \ell+c_0$.

As $\Delta$ is an egg, we can write $g=\d^1\dots\d^nh$ minimal representative with $\d^i \in \Delta$ and $h \in F$ (we use superscripts here only to avoid confusion with subscripts corresponding to sections). Then
\begin{eqnarray*}
g &=&\left( \prod_{i=1}^n\wrpr{\d^i_{x_1},\hdots,\d^i_{x_d}}{ \s^i} \right)\wrpr{h_{x_1}\hdots,h_{x_d}}{\s' }\\
&=&\left(\prod_{i=1}^n \wrprp{\d^i_{x_1},\hdots,\d^i_{x_d}}{\tau_{i-1}}{} \right)\wrprp{h_{x_1},\hdots,h_{x_d}}{\tau_n}{\tau_n\s'}, \quad \textrm{where} \quad \tau_i=\s^1\dots\s^i, \\
&=& \left(\prod_{i=1}^n \wrpr{\d^i_{\tau_{i-1}^{-1}.x_1},\hdots,\d^i_{\tau_{i-1}^{-1}.x_d}}{}\right)\wrpr{h_{\tau_n^{-1}.x_1},\hdots,h_{\tau_n^{-1}.x_d}}{\tau_n\s'} \\
&=& \wrpr{\left(\prod_{i=1}^n \d^i_{\tau_{i-1}^{-1}.x_1}\right)h_{\tau_n^{-1}.x_1},\hdots, \left(\prod_{i=1}^n \d^i_{\tau_{i-1}^{-1}.x_d}\right)h_{\tau_n^{-1}.x_d}}{ \tau_n \s'}.
\end{eqnarray*}
By injectivity of the embedding, we have $g_x=\left(\prod_{i=1}^n \d^i_{\tau_{i-1}^{-1}.x}\right)h_{\tau_n^{-1}.x}$ for each section $x$ and $\s=\tau_n\s'$. Therefore
\begin{eqnarray*}
\sum_{t=1}^d \|g_t\|_{S,\pi}&\le& \sum_{t=1}^d \sum_{i=1}^n \left\|\d^i_{\tau_{i-1}^{-1}.t}\right\|_{S,\pi}+\left\|h_{\tau_n^{-1}.t}\right\|_{S,\pi} \\
&=&\sum_{i=1}^n\sum_{t=1}^d\left\|\d^i_t\right\|_{S,\pi}+\sum_{t=1}^d \left\|h_t\right\|_{S,\pi} \le \sum_{i=1}^n \eta\left\|\d^i\right\|_{S,\pi} +c_0 \le \eta\|g\|_{S,\pi}+c_0\le \eta\ell+c_0.
\end{eqnarray*}
Note that $c_0$ is at most $d$ times the maximal length of a section of an element in $F$.
\end{proof}

\section{Semi-algorithms computing upper bounds}\label{sec:algo}


In order to apply the contraction criterion of the previous section and obtain an upper exponent $\alpha$ as small as possible, we have to find an egg $\Delta$ with minimal coefficient $\eta$ for some norm $\|\cdot\|_{S,\pi}$. More precisely, we want to estimate
\begin{eqnarray}\label{defetabar}
\bar{\eta}(G):=\inf_{(S,\pi)} \inf_{\D} \max_{w\in \D} \bar{\eta}_{S,\pi}(w),\quad \textrm{ where }\quad \bar{\eta}_{S,\pi}(w)= \frac{\sum_{x\in \XX}\|w_x\|_{S,\pi}}{\|w\|_{S,\pi}},
\end{eqnarray}
where the first infimum is over all  weighted generating sets $(S,\pi)$ and  the second is over all eggs $\Delta$ for~$\|\cdot\|_{S,\pi}$. However there is a last technical issue that it is difficult to estimate the actual norm $\|w\|_{S,\pi}$ of a word (or rather of the element represented by this word), and much easier to estimate the length $|w|_\pi$. So in practice we will evaluate
\begin{eqnarray}\label{defeta}
\eta(G):=\inf_{(S,\pi)} \inf_{\D} \max_{w\in \D} \eta_{S,\pi}(w),\quad \textrm{ where }\quad \eta_{S,\pi}(w)= \frac{\sum_{x\in \XX}|w_x|_{\pi}}{|w|_{\pi}},
\end{eqnarray}
This is sufficient because of the
\begin{fact}
$\bar{\eta}(G) \leq \eta(G)$. 
\end{fact}

\begin{proof}
Recall that we always have $\|w\|_{S,\pi}\le|w|_\pi$.
For each $w$ in a fixed egg $\D$, we have
\[
\bar{\eta}_{S,\pi}(w)\le \frac{\sum_{x\in \XX}|w_x|_{\pi}}{\|w\|_{S,\pi}}=\frac{\sum_{x\in \XX}|w_x|_{\pi}}{|w|_{\pi}}=\eta_{S,\pi}(w),
\] 
except if $\|w\|_{S,\pi}<|w|_\pi$ which means that $w$ is not minimal for $(S,\pi)$.  But in the latter case we can safely remove $w$ from $\D$ and still obtain an egg $\D'=\D \setminus\{w\}$.
\end{proof}

Finally if we have (\ref{defeta}), we deduce (\ref{defetabar}). Then Theorem \ref{MP} and Proposition \ref{prop-contract} give 
\[
\gamma_G(\ell)\le e^{c\ell^\a}, \quad \textrm{ where } \quad \a=\frac{\log(d)}{\log(d)-\log(\eta(G))} \textrm{ and }d=\#\XX.
\]

Automata groups come naturally with a preferred generating set (in bijection with the stateset). For simplicity, we have chosen to use this particular one and not change it.

\subsection{A semi-algorithm with fixed weights}

Let $(G,S,\pi)$ be fixed. We explain here how to obtain upper bounds on
\[
\eta(G,S,\pi):=\inf_{\D}  \eta_{S,\pi}(\D), \quad \textrm{ where } \quad  \eta_{S,\pi}(\D)=\max_{w\in \D}\eta_{S,\pi}(w).
\]
The first issue is to ensure the existence of an egg with $\eta_{S,\pi}(\D)<1$.

\subsubsection{A procedure to obtain eggs.}
Let $P(w)$ be a property of words in $S$, and assume that we are able to determine algorithmically whether $P(w)$ is true or not for a given $w$. 
We describe a procedure aiming at producing an egg $\D$ of $(G,S,\pi)$  such that $P(w)$ is true for all $w$ in $\D$.

At time one, we set \(L=S\) as a list. (We recall that \(e \notin S\).) At each time, we take the first word \(w\) in the list, we remove it from $L$ and do
\begin{enumerate}[label=(\alph*)]
\item if  \(\prop(w)\) is true, then add \(w\) to the set \(\Delta\),
\item if \(\prop(w)\) is false, then for each \(s \in S\)
\begin{enumerate}[label=(\roman*)]
\item if we detect that \(\|ws\|_{\pi}< \|w\|_{\pi}+\|s\|_{\pi}\), do nothing,
\item elif we detect $u$ in $L$ such that $ws=_Gu$, do nothing,
\item otherwise we add $ws$ at the end of $L$.
\end{enumerate}
\end{enumerate}

The procedure stops once $L$ is empty.

\begin{lemma}\label{lem-procegg}
If the procedure above stops in finite time, then the resulting set $\D$ is an egg of $(G,S,\pi)$ such that $P(w)$ is true for all $w$ in $\D$.
\end{lemma}

\begin{proof}
Let $g$ be an element of $G$ with a representative $w=s_1\dots s_m$ minimal  for $\|\cdot\|_{S,\pi}$. Arguing by induction, it is sufficient to prove that if $\|g\|_{S,\pi}$ is large enough, there is a prefix of $w$ which is equal in $G$ to a word $\d_1$ of $\D$.

Assume by contradiction that no prefix of $w$ is equal in $G$ to a word in $\D$. Then we claim that for each $k\ge 1$, there exists a word $w_k$ that was in $L$ at some time during the algorithm such that $g=_G w_ks_{k+1}\dots s_m$. 

Indeed, this is obvious for $k=1$. Assume this is true for $k$ and consider $w_ks_{k+1}$. This word is not in $\D$ by assumption, so case (a) does not apply. It is of length $\|w_k\|_{S,\pi}+\|s_{k+1}\|_{S,\pi}$ by minimality of $w$, so case (b)(i) does not apply. Therefore either case (b)(ii) $w_k s_{k+1}$ equals in $G$ to some $u=:w_{k+1}$ that was in $L$ at the same time, or case (b)(iii) $w_ks_{k+1}=:w_{k+1}$ was added to $L$. 

If the algorithm stopped in finite time, there has been finitely many elements in $L$ and we get a contradiction as soon as $|w_k|_{\pi}$ is bigger than all their $\pi$-lengths.
\end{proof}

In fact the lemma still holds if we do not try to detect (b)(i) or (b)(ii). However the resulting egg $\Delta$ (if it exists) would be bigger and the time of computation would be larger. 


\subsubsection{The target semi-algorithm.}\label{sec:com}

Given an automaton group $G$ with a fixed weighted generating system $(S,\pi)$, we use a semi-algorithm, presented in \cref{fig:procegg}, to obtain an upper bound on $\eta(G,S,\pi)$. This semi-algorithm implements the procedure described above with the following property.

For a target coefficient $\eta_\tar \in (0,1]$, we say a word $w$ satisfies property $P_{\eta_\tar}(w)$  if it satisfies condition (\ref{eqn:contract}) of Proposition \ref{prop-contract} for $\eta=\eta_\tar$, i.e. 
\[
\eta_{S,\pi}(w)=\frac{\sum_{x\in \XX}|w_x|_{\pi}}{|w|_{\pi}}\le\eta_\tar
\]
Lemma \ref{lem-procegg} guarantees that if the semi-algorithm stops, it furnishes an egg for $\|\cdot\|_{S,\pi}$ all of which element have property $P_{\eta_\tar}(w)$.

In order to try to detect (b)(i) when $\|ws\|_{S,\pi} <\|w\|_{S,\pi}+\|s\|_{S,\pi}$ we consider an auxiliary group $\til{G}$ of which $G$ is a quotient (so minimality in $G$ implies minimality in $\til{G}$ but not conversely). Then $\|\cdot\|_{\til{G},\til{S}}$ just denotes the usual word norm with respect to $\til{S}$ as in \cref{sec:stabquo}.  In practice, this auxiliary group will be either a free group or a finite free product of free groups.

We also use an auxiliary list $L_{\textrm{aux}}$. All the elements of this auxiliary list have the same $\|\cdot\|_{\til{G},\til{S}}$-length, and the evolution of the size of this list (with length of words) gives a good heuristical information whether the procedure is about to succeed or fail.

Note that according to (b)(ii), we add the word $ws$ to the auxiliary list  $L_{\textrm{aux}}$ only if it represents a new group element. This is what we mean by $\cup_G$. We are aware that this uncontrolled choice of representative word (a priori) impacts the final bound but it would be substantially more difficult to take this into account and seek the best representative.

Finally, the egg obtained obviously satisfies that $\eta_{S,\pi}(w)\le \max_{w\in\D}\eta_{S,\pi}(w)=:\eta_{\max}$ for all $w$ in $\D$, which we can record along the run. We conclude
\[
\eta(G,S,\pi):=\inf_{\D}  \eta_{S,\pi}(\D) \le \eta_{\max}.
\]

This algorithm was implemented in \texttt{GAP} by the function \texttt{IsSubExp\_rec}. It is also performed by \texttt{IsSubExp\_fast} which is faster (it avoids functions calling). When the weights are uniform, the function \texttt{IsSubExp\_fast\_uni} is even faster. The code is given at \texttt{https://www.irif.fr/\textasciitilde godin/automatongrowth.html}. There the list $\Delta$ is called \emph{shell} and the list $L$ is called the \emph{yolk}. The algorithm stops once the yolk list is empty. While running the algorithm, we record their sizes which give a good heuristical indication whether the algorithm is likely to finish or not. 
Moreover in order to avoid endless computation, we input a \emph{radius}, which is the maximal  length (uniform in $S$) of words allowed in the shell list $\Delta$ and the yolk list $L$.

\begin{figure}
  \begin{center}
\begin{algorithm}[H]
    \SetAlgoLined
    \KwData{An automaton group \(G,S,\pi\) and a target coefficient $\eta_\tar \in (0,1]$.}
    \KwResult{An egg \(\D\) and a coefficent $\eta_{\max}\leq \eta_\tar$ such that $\eta_{S,\pi}(w) \le\eta_{\max}$ for all $w\in\D$}
\BlankLine 
    \(L \leftarrow S\) \; $\eta_{\max} \leftarrow 0$\;
 
    \While{ \(L \neq \emptyset \)}{

	    \(L_{\textrm{aux}} \leftarrow \emptyset \)\;
        \For{ \(w \in L\)}{
        \eIf{\(\eta(w)\le \eta_\tar\)}{
 		\(\D \leftarrow \D \cup \{w\}\)\; $\eta_{\max} \leftarrow \max\left(\eta_{\max},\eta_{S,\pi}(w)\right)$\;
            }{
            \For{ \(s \in S\)}{
            	\If{ \(|ws|_{\til{G},\til{S}}= |w|_{\til{G},\til{S}}+|s|_{\til{G},\til{S}}\)}{
            		\(L_{\textrm{aux}} \leftarrow L_{\textrm{aux}} \cup_G \{ws\}\)\;
            	}
            }
        }
        }
        \(L \leftarrow L_{\textrm{aux}} \)\;
    }

\end{algorithm}

  \end{center}
\caption{The target semi-algorithm of Section \ref{sec:com}, implemented as \texttt{IsSubExp\_rec} in \texttt{GAP}.}\label{fig:procegg}
\end{figure}

\subsection{Improving weights for a fixed egg}\label{sec:wmo}

Once we obtain an egg, we wish to modify the weights in order to minimize $\eta_{S,\pi}(\Delta)$. The difficulty is twofold. 
First by construction, the fact that a set is an egg depends a priori strongly on the weights chosen. Second we need to actually minimize a function of the form $\eta_{S,\pi}(\D)=\max_{w\in \D}\eta_{S,\pi}(w)$, where $\eta_{S,\pi}(w)$ is a rational function of $\pi\in (0,1]^S$.

The choice of auxiliary group $\til{G}$ permits to bypass the first difficulty. Indeed, we choose $\til{G}$ to be either a free group or a finite free product of finite groups. In the first case, we impose no restriction on $\pi$ while in the second case we impose that $\pi$ satisfies the triangular conditions (\ref{cond}) of Lemma \ref{minfred}. For such weights in such groups, minimality of representative words is invariant, so the detections (b)(i) when we perform the target semi-algorithm do not depend on the choice of weight. This guarantees that we can change weights freely as long as we respect the triangular conditions.

There remains the second task to evaluate the following function of $\pi$:
\[
\eta(\D)=\inf_{\pi \in \Omega} \max_{w \in \D}\eta_\pi(w), \quad \textrm{ where }\quad \eta_{\pi}(w)= \frac{\sum_{x\in \XX}|w_x|_{\pi}}{|w|_{\pi}}.
\]
The domain $\Omega$ is a subset of $[0,1]^S$ bounded by the triangular conditions (\ref{cond}) and the functions $\eta_\pi(w)$ are rational functions of $\pi$. Indeed for a word $w$ in $S$ and an element $s \in $S, let us denote by $N_s(w)$ the number of occurences of the letter $s$ in $w$. Then $|w|_{\pi}=\sum_{s \in S} N_s(w)\pi(s)$.
It follows that
\[
\eta_{\pi}(w)=\frac{\sum_{s \in S}\sum_{x\in\XX}N_s(w_x) \pi(s)}{\sum_{s \in S} N_s(w)\pi(s)}.
\]
When we run the target semi-algorithm, we register the coefficients $N_s(w)$ and $N_s(w_x)$ for $s\in S$ and $x \in \XX$.

To concretely estimate $\eta_\pi(\D)$, we use a numerical algorithm \texttt{wmo} based on a generalized gradient method. \texttt{wmo} comes from recursive multi-level solution of a second order 
dynamical system with embeded ad-hoc inertia effects to escape from local minima. The algorithm is described in~\cite{Moh07,MoRe09}. In implementation, the overall \texttt{GAP}-program calls for the application \texttt{wmo} providing it with the coefficients above and the triangular constraints. The algorithm \texttt{wmo} returns the estimated minimum together with the estimated optimal weights. 

It seems very challenging to make exact optimization here because the function to minimize is not convex and has several local minima. It also appeared experimentally on some sets of data that very different weights $\pi$ can lead to values extremely close to the observed minimum.

\subsection{Mixed strategies}

The minimization problem (\ref{defeta}) is intrinsically two dimensionnal, since we have to find both an optimal weight $\pi$ and an optimal egg $\Delta$. (Recall that in order to simplify the problem, we chose to fix once and for all the generating set $S$.) The algorithms described previously are essentially one dimensionnal, and they have to be coupled to obtain good results. For this, we implemented three \texttt{GAP}-functions described at \texttt{https://www.irif.fr/\textasciitilde godin/automatongrowth.html}.
\begin{itemize}
\item The function \texttt{IsSubExp\_opt} takes as input a given weight $\pi$ and a target $\eta_\tar$, it searches an egg according to the target semi-algorithm. Once it obtains an egg $\D_1$ it calls \texttt{wmo} in order to find the weights $\pi_1$ optimizing $\eta(\D_1)$. This algorithm can be looped by searching for a new egg $\Delta_2$ using the weights $\pi_1$ for a new target $\eta_{\tar,2}$. 
\item Given an update parameter, the function \texttt{IsSubExp\_ovi} takes as input a given weight $\pi$ and a target $\eta_\tar$, it searches an egg according to the target semi-algorithm, but everytime the yolk list contains words of length multiple of the update parameter, it applies the weight-optimization algorithm \texttt{wmo} to the whole set of words $\Delta \cup L$, and updates the weight vector according to the output. It stops once it reaches the target.
\item The function \texttt{IsSubExp\_loop} permits the user to control the strategy. At each step the function asks for an instruction either of exploration (applying the main loop of the target algorithm of \cref{fig:procegg} to a chosen subset of $\Delta \cup L$) or of  weight optimization (using \texttt{wmo} over a chosen subset of $\Delta \cup L$). 
\end{itemize}

\section{Data and comments}\label{sec:data}


\subsection{Test of implemented functions} We use the first Grigorchuk group in order to test our functions. We consider it together with its usual Mealy automaton, but also with the associated automata of the actions on the second and third levels. In these disguised descriptions, it is much harder to obtain good numerical upper bounds. 

\subsubsection{Test of \texttt{IsSubExp\_rec}}
In view of \cite{Bar98,ErZh18}, we know the optimal weights are $[.305061,.34747,.223839,.123631]$ for $S=\{a,b,c,d\}$ and the optimal contraction coefficients $\eta$ are $\eta_{\textrm{opt}}=.8106$,  $\eta_{\textrm{opt}}^2=.6572$ and $\eta_{\textrm{opt}}^3=.5327$ for  level $1$, level $2$ and level $3$ respectively. 

Indeed, iterating the condition of \cref{MP} ensures that
\[
B_{G,S,\pi}(\ell) \hookrightarrow \bigcup_{\ell_1+\dots+\ell_{d^2}\le \eta^2 \ell+c_1} K^{d+1} \times \prod_{i=1}^{d^2} B_{G,S,\pi}(\ell_i),
\]
so if $\eta$ is valid for the first level, then $\eta^2$ is valid for the second level. By induction $\eta^k$ is valid also for the $k$th level. However, a computationnal issue is the increase of the additive constant $c_1=\eta c_0+(d+1)c_0$. This additive constant is unessential in theory, but in practice it has to be taken care of by an over approximation of $\eta$ in the ratios (\ref{eqn:contract}) of \cref{prop-contract}. This explains why the radius and size of eggs blow up as we tighten our precision on the target. We explored numerically this dependance on the target for the optimal weights in \cref{table:recGri-piopt}.

\begin{table}[h]
\centering
\begin{tabular}{|c|c|c|c|c|c|}
\cline{1-6}
level & target &radius &  egg size &$\eta$ & $\alpha$  
\\ \cline{1-6}
1 & .90 & 2 & 4 & .8106 & .7675 
\\ \cline{1-6}
\multirow{4}{*}{2} & .9 & 8 & 12 & .8857  & .9195
\\ \cline{2-6}
 & .75 & 14 & 73 & .7497  &  .8280
\\ \cline{2-6}
& .70 & 25 & 1360 & .6997  & .7952
\\ \cline{2-6}
& .69 & 33 & 5947 & .6900  & .7889
\\ \cline{2-6}
& .68 & 45 & 93855 & .6800  & .7824
\\ \cline{1-6}
\multirow{7}{*}{3} & .9 & 9 & 32 & .8954  & .9496
\\ \cline{2-6}
 & .80 & 11 & 52 & .7980  &  .9022
\\ \cline{2-6}
& .70 & 17 & 154 & .6999  & .8536
\\ \cline{2-6}
& .65 & 23 & 427 & .6493  & .8281
\\ \cline{2-6}
& .60 & 36 & 3715 & .5999  & .8028
\\ \cline{2-6}
 & .59 & 40 & 7499 & .5900  &  .7977
\\ \cline{2-6}
& .58 & 46 & 19616 & .5800  & .7924

\\ \cline{1-6}

\end{tabular}
	\caption{Testing the function \texttt{IsSubExp\_rec} on the first Grigorchuk group with optimal weights.}\label{table:recGri-piopt}
\end{table}

We also studied the dependance on the weights, using perturbations of the (normalized) optimal weights by vectors $\lambda[1,-1,1,-1]$ for $\lambda=.02$ and $\lambda=.005$. The data are given in \cref{table:recGri-pert}. 

It appears that in some situations the non-optimal weights provide (very slightly) smaller eggs, but as we tighten our target precision, this is no longer the case. For this reason, it will be difficult to detect whether weights are optimal or not.

\begin{table}[h]
\centering
\begin{tabular}{|c|c|c|c|c|c|c|}
\cline{1-7}
weights on first Grigorchuk group&level & target &radius &  egg size &$\eta$ & $\alpha$  
\\ \cline{1-7}
\multirow{7}{*}{$\begin{array}{c}[.325061,.32747,.243839,.103631]\\ \textrm{optimal } \pm .02\end{array}$} &1 & .90 & 2 &  4 &.8719 & .8349  
\\ \cline{2-7}
 &\multirow{2}{*}{2} & .75& 17 &  88 & .7471 & .8263
\\ \cline{3-7}
 & & .70& 89 &  28645 & .7000 & .7954
\\ \cline{2-7}
 &\multirow{4}{*}{3} & .65& 25 &  418 & .6497 & .8283
\\ \cline{3-7}
 & & .60& 40  & 4145   & .6000  & .8028
\\ \cline{3-7}
 & & .59& 53  & 9672   & .5900  & .7977
\\ \cline{3-7}
 & & .58& 69  & 33001   & .5800  & .7925
\\ \cline{1-7}

\multirow{7}{*}{$\begin{array}{c}[.310061,.34247,.228839,.118631]\\ \textrm{optimal } \pm .005\end{array}$} &1 & .90 & 2 &  4 &.8261 & .7840  
\\ \cline{2-7}
 &\multirow{2}{*}{2} & .75& 15 &  76 & .7419 & .8229
\\ \cline{3-7}
 & & .70& 29 &  1595 & .7000 & .7954
\\ \cline{2-7}
 &\multirow{4}{*}{3}& .65& 23 &  427 & .6499 & .8284
\\ \cline{3-7}
 & & .60& 35 &  3576 & .6000 & .8028
\\ \cline{3-7}
 & & .59& 40 &  7477 & .5900 & .7977
\\ \cline{3-7}
 & & .58& 50 &  20744 & .5800 & .7925
\\ \cline{1-7}
\end{tabular}
	\caption{Dependance of radius and size of eggs on the weights -- compare also with \cref{table:recGri-piopt}.}\label{table:recGri-pert}
\end{table}

\subsubsection{Test of \texttt{IsSubExp\_ovi}}

Given a target, this function calls \texttt{wmo} in order to optimize weights once an egg for the initial weights is obtained. We would like to use it in order to find suitable weights for new automata groups where optimal weights are unknown. Data obtained for the Grigorchuk group is described in \cref{table:oviGri}

A first flaw is that the existence of such an egg depends on the weights. For instance for uniform weights on the Grigorchuk group at first or second level, no egg is obtain so the algorithm does not end. The reason is that words of the form $(ab)^k$ are not contracting. 
An easy way to bypass this difficulty is to break the symmetry and use almost uniform weights. However not all symmetry breaking works.

A second flaw appears for instance on level two with almost uniform initial weights. The output weights turn out to be a permutation of the optimal weights and when we plug in these weights with a tighter  target, the algorithm does not end. We reached a local minimum of our optimization problem and this function cannot exit it.

When given aa initial weight close to the optimal value (say $[.30,.35,.20,.10]$), the function outputs a slightly better weight. But as seen on \cref{table:oviGri} the evolution of the weights and contraction coefficient $\eta$ is very slow and it is not clear whether they would often converge to the optimal vector or to unsatisfying local minima.

\begin{table}[h]
\centering
\begin{tabular}{|c|c|c|c|c|c|c|c|c|}
\cline{1-9}
level & loop & target& initial weights  & radius & egg size & $\eta$ & $\alpha $ & output weights
\\ \cline{1-9}

\multirow{3}{*}{1} & 1 &.9999& uniform & \multicolumn{5}{c|}{does not end} 
\\ \cline{2-9}
 & 1& .9999 &  [1.,.99,.98,.97] & 2 & 4 & .8119 & .7689 & [.3067,.3467,.2232,.1236]
\\ \cline{2-9}
 & 1& .9999 &  [1.,.97,.98,.99] & \multicolumn{5}{c|}{does not end} 
\\ \cline{1-9}

\multirow{6}{*}{2} & 1 &.9999& uniform & \multicolumn{5}{c|}{does not end} 
\\ \cline{2-9}
 & 1& .9999 &  [1.,.99,.98,.97] & 5 & 8 & .8107 & .8686  & [.3052,.3474,.1237,.2239]
\\ \cline{2-9}
 & 2& .8106 &  from loop 1 &  \multicolumn{5}{c|}{does not end} 
\\ \cline{2-9}
 & 1& .80 &  [.30,.35,.20,.10] & 10 & 30 & .7901 & .8548 & [.2851,.3575,.2251,.1325]
\\ \cline{2-9}
 & 2& .75 & from loop 1  & 14 & 93 & .7463 & .8258 & [.2866,.3568,.2257,.1311]
\\ \cline{2-9}
 & 3& .70 & from loop 2  & 30  & 1700  & .6999 & .7953 & [.2865,.3568,.2258,.1312]
\\ \cline{1-9}

\multirow{11}{*}{3} & 1 &.9999& uniform & 9 & 32 & .8508 & .9279 & [.3164,.3419,.1710,.1710]
\\ \cline{2-9}
 & 2 &.85& from loop 1 & 10& 36 & .7992 & .9027 & [.3563,.3219,.1917,.1303]
\\ \cline{2-9}

 & 3 & .70 & from loop 2 & 17 & 147 & .6910 & .8491 & [.3720,.3141,.1932,.1209]
\\ \cline{2-9}

& 4 & .65 & from loop 3 & 23 & 361 & .6491 & .8280 & [.3736,.3132,.1934,.1199]
\\ \cline{2-9}

& 5 & .63 & from loop 4 & 30  & 810  & .6281  & .8173 & [.3696,.3153,.1956,.1198]
\\ \cline{2-9}

& 6 & .61 & from loop 5 & 42  & 2155   & .6099   & .8079 & [.3687,.3157,.1968,.1190]
\\ \cline{2-9}
& 7 & .60 & from loop 6 & 48  & 3964   & .6000   & .8028 & [.3682,.3159,.1969,.1191]
\\ \cline{2-9}

& 1 &.7& [.40,.30,.10,.20] & 40 & 315  & .6971  & .8522 &  [.4024,.2988,.1035,.1955]
\\ \cline{2-9}
& 2 &.67& from loop 1 & 166 & 2313 & .6672 & .8371 & [.4097,.2952,.1107,.1846]
\\ \cline{2-9}
& 3 &.66& from loop 2 & 114 & 2636 & .6587 &.8328 & [.4023,.2989,.1178,.1812]
\\ \cline{2-9}
& 4 &.65& from loop 3 & 234 & 4854&.6492 &.8280& [.3998,.3001,.1195,.1808]
\\ \cline{2-9}
 & 1& .70 &  [.30,.35,.20,.10] & 16 & 154 & .6966 & .8519 & [.2943,.3529,.2286,.1244]
\\ \cline{2-9}
 & 2& .65 &  from loop 1 & 23 & 535 & .6497 & .8283 & [.2966,.3517,.2275,.1243]
\\ \cline{1-9}

\end{tabular}
	\caption{Testing the function \texttt{IsSubExp\_ovi} on the first Grigorchuk group.}\label{table:oviGri}
\end{table}

\subsubsection{Test of \texttt{IsSubExp\_opt}}
Given a fixed target, this function looks for an egg, but every update round it applies the optimizer \texttt{wmo} to the whole set of words $\D\cup L$, then it continues using the new weights until next update. Experimental data are given in \cref{table:optGri}.

At the first level, the results and output weights are very close to optimal, but it may be due to the simplicity of the problem. At the second level, they are still close to optimal but the obtained eggs are substantially bigger. Again, the relative simplicity of the problem might account for this success.

At the third level, it is much harder to obtain satisfying output vectors. The best results are obtained with tighter targets and sparser updates, but at the expense of longer computations.

Note that on level $3$, for target $.65$ and update $4$, at round $28$ the use of \texttt{wmo} provided the following vector $[.2959,.3521,.2705,.0817]$, much closer to the optimal vector than the final output -- see \cref{table:optGri}.

\begin{table}[h]
\centering
\begin{tabular}{|c|c|c|c|c|c|c|c|}
\cline{1-8}
level &  target& update  & radius & egg size & $\eta$ & $\alpha $ & output weights 
\\ \cline{1-8}
\multirow{3}{*}{1} & \multirow{3}{*}{.90} & 2   & 2 & 4 & .8113 & .7683 & [.3052,.3475,.2243,.1232]
\\ \cline{3-8}
 & &  4  & 5 &8  & .8107 & .7676 & [.3053,.3474,.2238,.1236]
\\ \cline{3-8}
 & &  10  &  11& 64 & .8136 & .7707 &  [.3118,.3442,.2219,.1224]
\\ \cline{1-8}

\multirow{8}{*}{2} & \multirow{3}{*}{.90} & 2   &  2 & 4 & .8113 & .8690 & [.3100,.3452,.1227,.2224]
\\ \cline{3-8}
 & &  4  & 4 & 8 & .8121 & .8695 &  [.3072,.3465,.1229,.2236]
\\ \cline{3-8}
 & &  10  & 10  &  32 &  .8124 & .8697 &  [.3089,.3456,.1232,.2534]
\\ \cline{2-8}

 & \multirow{3}{*}{.75} & 2   &  16 & 284 & .7467 & .8260 & [.3071,.3465,.1912,.1554]
\\ \cline{3-8}
 & &  4  & 18 & 315 & .7487 &.8273  &  [.3129,.3436,.1922,.1515]
\\ \cline{3-8}
 & &  10  & 24  &  536 &  .7499 & .8281 &  [.3296,.3353,.1822,.1532]
\\ \cline{2-8}

 & \multirow{2}{*}{.72} &  4  & 32  & 805 & .7166 & .8063 &  [.3069,.3466,.2068,.1399]
\\ \cline{3-8}
 & &  10  & 36  & 1199  & .7198  & .8083 &  [.3051,.3475,.2086,.1390]
\\ \cline{1-8}

\multirow{11}{*}{3} & .90 & 2 or 4 or 10   &  9 & 24 & .8889 & .9464 & [.2500,.2500,.2500,.2500]
\\ \cline{2-8}

 & \multirow{4}{*}{.70} & 2   &  23& 269 & .6899 & .8486 & [.3762,.3119,.2223,.0897]
\\ \cline{3-8}
 & & 4   & 21  & 225 & .6985 & .8529 & [.4729,.2636,.1814,.0823]
\\ \cline{3-8}
 & &  10  & 28  &  231 &  .6990 & .8531 &  [.5089,.2456,.2229,.0228]
\\ \cline{3-8}
 & &  15  & 23  & 462  &  .6996 & .8534 &  [.2501,.3750,.2438,.1313]
\\ \cline{2-8}
 & \multirow{4}{*}{.65} & 2   & 40  & 715 & .6430 & .8249 & [.4901,.2550,.1847,.0704]
\\ \cline{3-8}
 & & 4   & 40  & 934 & .6477 & .8273 & [.4789,.2606,.2059,.0548]
\\ \cline{3-8}
 & &  10  & 38  &  670 &  .6487 & .8278 &  [.4597,.2702,.1877,.0826]
\\ \cline{3-8}
 & &  15  & 24  &  961 &  .6499 & .8284 &  [.2795,.3603,.2443,.1161]
\\ \cline{2-8}

 & \multirow{2}{*}{.63} & 10   &60 &1823&.6287& .8176 &[.4004,.2999,.2345,.0654]
\\ \cline{3-8}
 & & 15   & 27 & 1734 & .6300&.8183&[.2795,.3603,.2443,.1161]
\\ \cline{1-8}
\end{tabular}
	\caption{Testing the function \texttt{IsSubExp\_opt} on the first Grigorchuk group. Initial weights are uniform.}\label{table:optGri}
\end{table}

\FloatBarrier

\subsection{Results on new groups}

\subsubsection{The group of \cref{fig:T1} acting on a $6$-letters alphabet}

The function \texttt{IsSubExp\_opt} starting from uniform weights with $\eta_\tar=.67$ and update $10$ yield the vector $[.3352,.1899,.1899,.2849]$. 
With these weights, the function \texttt{IsSubExp\_fast} attained target $\eta=0.645$ with radius $76$ and an egg of size $69978$. So this group has growth exponent at most $\alpha=.8034$.
In view of the difficulty to find optimal weights, it is likely that the actual exponent is substantially smaller. However it would be a surprise if it were smaller than that of the Grigorchuk group.

\subsubsection{The group of \cref{fig:mNote} acting on an $8$-letters alphabet}

Optimizations stronly suggest to use the following weights $[1.,0.,0.]$, i.e. to take only the involutive generator into account. Then the function \texttt{IsSubExp\_fast} attained target $\eta=0.819$ with radius $464$ and an egg of size $2098$. So this group has growth exponent at most $\alpha=.9124$. We believe this bound is close to optimal.

Note that using uniform weights, \texttt{IsSubExp\_fast} attained target $\eta=0.8299$, hence $\alpha=.9178$, with radius $157$ and an egg of size $5690$.

\subsubsection{The group of \cref{fig:Y} acting on a $7$-letters alphabet}

The function \texttt{IsSubExp\_opt} with $\eta_\tar=.9$ and update $10$ returned the weights $[.3115,.2731,.4154]$. 
With this and target $.83$, the function \texttt{IsSubExp\_ovi} yield the weights $[.3238,.2794,.3968]$ and $\eta=.8297$. Then \texttt{IsSubExp\_fast} attained target $\eta=.8200$ with radius $91$ and an egg of size $54727$. So this group has growth exponent at most $\alpha=.9075$.


\subsubsection{The group with $X$-shape Schreier graph, $9$ states and $17$ letters}
This group was defined in \cref{sec:examples}.
The function \texttt{IsSubExp\_fast\_uni} attained target $\eta=0.9286$  with radius $18$ and an egg of size $240039$. So this group has growth exponent at most $\alpha=.9746$. It is likely that the actual growth exponent is much smaller, but it is already challenging to prove intermediate growth. The amount of data is so big that the functions \texttt{IsSubExp\_ovi} and \texttt{IsSubExp\_opt} cannot get to optimization on an ordinary laptop. This is why only uniform weights are used.


\section{Super-polynomial growth}\label{sec:superpoly}


In this section we provide a folklore sufficient criterion that ensures super-polynomial growth of some automata groups. In particular, all the assumptions below are satisfied by the examples presented in \cref{sec:examples}.

 Let $G$ be an automaton group. 

\begin{assumption}\label{hyp1}
There exists $A_1,\dots ,A_k$ finite subgroups of $G$ with pairwise trivial intersection such that
\begin{itemize}
\item either $k\ge 3$ or $k=2$ and $A_1,A_2$ are not both groups of size $2$,
\item the group $G$ is generated by an automaton $\auta$ with stateset $Q=A_1\cup\dots\cup A_k$.
\end{itemize}
\end{assumption}
Under \cref{hyp1}, we take $S=Q\setminus\{e\}$. We denote the self-similar images of the generators~as:
\begin{equation}\label{eq:S}
\forall s \in S, \quad \psi(s)=\wrpr{\sect{s}{x_1},\hdots, \sect{s}{x_d}}{\perm{s}}.
\end{equation}

\begin{assumption}\label{hyp2}
For any word $w$ in $S^*$, and for any section $x \in \XX$, we have
\[
|\sect{w}{x}|_S \leq \frac{|w|_S+1}{2}.
\]
\end{assumption}

Given our automaton $\auta$, we construct an approximating sequence $(G^k)_{k \ge 0}$ as follows (we use superscripts to avoid confusion with sections which are denoted with subscripts). 

Set $G^0=A_1*\dots*A_k$ to be the free product of the finite subgroups  of \cref{hyp1} and $S^0=S$.

For each $k\ge 1$, the group $G^k$ will be a subgroup of $G^{k-1} \wr_\XX \Sym(\XX)$ generated by a set $S^k$ in canonical bijection with $S$. More precisely
\[
\psi^k:G^k=\langle S^k\rangle \hookrightarrow G^{k-1} \wr_\XX \Sym(\XX)
\]
where the generators are given by 
\[
\forall s^k \in S^k, \quad \psi^k(s^k)=\wrpr{\sect{s^{k-1}}{x_1},\hdots, \sect{s^{k-1}}{x_d}}{\perm{s}} \quad \textrm{modelled on (\ref{eq:S})}.
\]
(Here, $S^k, S^{k-1}$ and $S$ are all in canonical bijection as generating sets of their appropriate groups.)

\begin{lemma}\label{lem:rel}
Under Assumptions\ref{hyp1} and \ref{hyp2}, for any word $w$ in $S^*$ of length $|w|_S \le 2^k$, the evaluation $w(S)$ in $G$ is trivial if and only if the evaluation $w(S^k)$ in $G^k$ is trivial.
\end{lemma}

\begin{proof}
To check whether the evaluations are trivial, it is sufficient to check if their images in the iterated wreath products 
\[
G^k \overset{\psi^k\circ\dots\circ\psi^1}{\hookrightarrow} G^0\wr_{\XX^k}\left(\Sym(\XX)\wr_\XX \dots \wr_\XX \Sym(\XX) \right) \textrm{ and } G \overset{\psi^{\circ k}}{\hookrightarrow} G\wr_{\XX^k}\left(\Sym(\XX)\wr_\XX \dots \wr_\XX \Sym(\XX) \right)
\]
are trivial. But by construction of $G^k$, the images in $\Sym(\XX)\wr_\XX \dots \wr_\XX \Sym(\XX)$ are the same. On the other hand for any section $\mot{u} \in \XX^k$, one has $\sect{w(S^k)}{\mot{u}}=\sect{w(S)}{\mot{u}}$ as words in $S^0$ and $S$ respectively. Now Assumption~\ref{hyp2} guarantees that if $|w|_S \le 2^k$, then $|\sect{w}{\mot{u}}|_S \le 1$. By Assumption~\ref{hyp1} and definition of $G^0$, the evaluations of $\sect{w(S^k)}{\mot{u}}$ and $\sect{w(S)}{\mot{u}}$ are either both trivial or both non-trivial. 
\end{proof}

\begin{assumption}\label{hyp3}
The self-similarity map $\psi:G \hookrightarrow G \wr_\XX \Sym(\XX)$ is surjective on the first section, i.e. for all $g_1$ in $G$, there exists $g$ in $G$ such that $\sect{g}{x_1}=g_1$.
\end{assumption}

\begin{proposition}
Assume $G$ satisfies Assumptions \ref{hyp1}, \ref{hyp2}, \ref{hyp3} and (\ref{eqn:contract}) of \cref{prop-contract}, then $G$ has intermediate growth.
\end{proposition}

The following proof uses the notion of amenability of a group, which we will not define here. The reader can refer to \cite{Bartholdi2017survey} for instance. For our purpose, it is sufficient to know that subexponential growth implies amenability, that amenability passes to quotients and subgroups and that virtually free groups (e.g. the group $G^0$ under Assumption~\ref{hyp1}) are not amenable.

\begin{proof}
Property (\ref{eqn:contract}) of \cref{prop-contract} guarantees subexponential growth. There remains to check super-polynomial growth. By Gromov's theorem, this amounts to prove that $G$ is not virtually nilpotent. As finitely generated nilpotent groups are finitely presented, it is sufficient to prove that $G$ is not finitely presented. Assume the contrary and take a finite presentation of $G$ with respect to the generating set $S$. There exists some $R$ such that all the defining relations have length $\le R$. By \cref{lem:rel} all these relations also hold in $(G^k,S^k)$ for $k$ large enough. Therefore $G^k$ is a quotient of $G$ hence amenable too. However by construction and Assumption~\ref{hyp3}, the group $G^k$ contains a non-trivial free product $G^0=A_1*\dots*A_k$ which is non amenable. This is a contradiction.
\end{proof}

\bibliographystyle{abbrv}
\bibliography{growth}

\textsc{\newline J\'er\'emie Brieussel --- IMAG, Universit\'e de Montpellier 
} --- jeremie.brieussel@umontpellier.fr

\textsc{\newline Thibault Godin --- IMAG, Universit\'e de Montpellier 
} --- thib.godin@gmail.com

\textsc{\newline Bijan Mohammadi --- IMAG, Universit\'e de Montpellier 
} --- bijan.mohammadi@umontpellier.fr

\newpage

\appendix

\section{Figures of automata groups and their Schreier graphs}\label{app:fig}

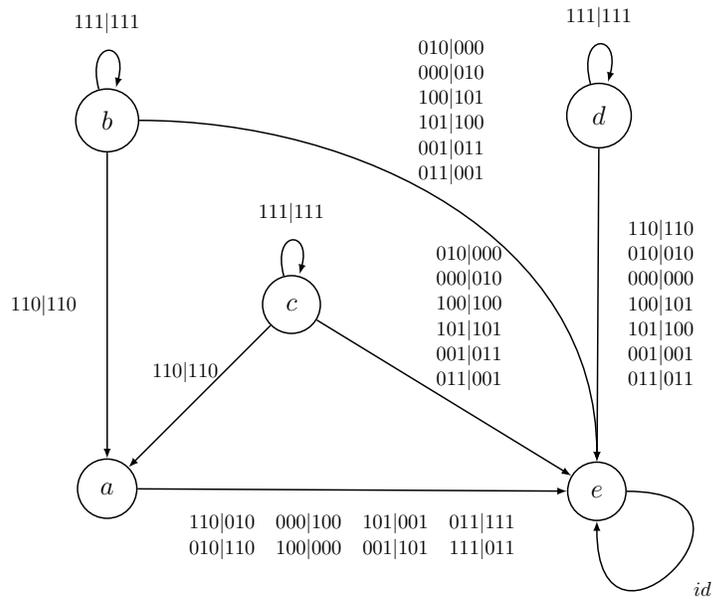
\begin{figure}[h!]

\begin{center}

{
\scalebox{.7}{
{
	\begin{tikzpicture}[->,>=latex,node distance=35mm,thick, inner sep =8pt]

		\node[state, inner sep =8pt] (c)  {\Large $c$};
				\node[state, inner sep =8pt] (a) [left of=c,below of=c] {\Large $a$};
				\node[state, inner sep =8pt] (b) [left of=c, above of=c] {\Large $b$};
				\node[state, inner sep =8pt] (d) [above right= 27.5mm and 50mm of c ] {\Large $d$};
				\node[state, inner sep =8pt] (e) [below right= 27.5mm and 50mm of c ] {\Large $e$};
			\path
				(d) edge[loop above] node[above]{\(\begin{array}{c}  111|111 \end{array}\)}	(d)
				(c) edge[loop above] node[above]{\(\begin{array}{c}  111|111 \end{array}\)}	(c)				
				(b) edge[loop above] node[above]{\(\begin{array}{c}  111|111 \end{array}\)}	(b)				
				(a)	edge	node[below=.1cm]{\(\begin{array}{c} 110|010 \quad 000|100  \quad  101|001 \quad 011|111\\010|110 \quad 100|000  \quad  001|101 \quad 111|011 \end{array} \)}	(e)
				(b) edge 	node[left=.1cm]{\(\begin{array}{c} 110|110 \\ \end{array}\)} (a)
				(c) edge 	node[above=.2cm, pos=0.6]{\(\begin{array}{c} 110|110 \end{array}\)} (a)
				(b) edge[out=00,in=90] 	node[above=.1cm]{\(\begin{array}{c} 010|000 \\000|010 \\ 100|101 \\ 101|100\\ 001|011 \\011|001 \end{array}\)} (e)
				(c) edge 	node[above=.1cm, pos=0.6]{\(\begin{array}{c}  010|000 \\000|010 \\ 100|100 \\ 101|101\\ 001|011  \\011|001\end{array}\)} (e)
				(d) edge 	node[right=.1cm]{\(\begin{array}{c}110|110 \\ 010|010 \\000|000 \\ 100|101 \\ 101|100\\ 001|001 \\011|011\end{array}\)} (e)
				(e) edge[min distance=5mm,out=00,in=270,looseness=10] node[below right =-1mm and -1mm]{\(\begin{array}{c} id \end{array}\)}	(e)
		
		;
	\end{tikzpicture}
	}
	}	
}
\end{center}
\caption{The automata generating the first Grigorchuk group on level 3.}\label{fig:Gri3}
\end{figure}

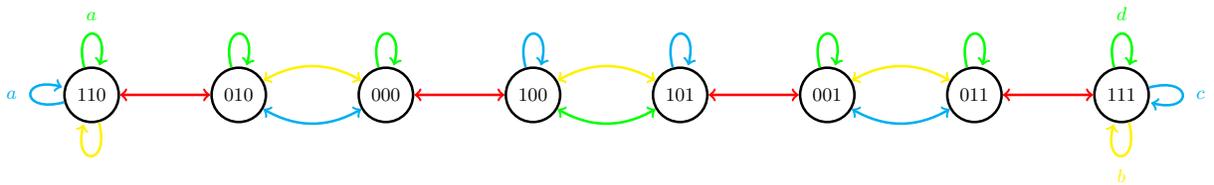
\begin{figure}[h!]

\begin{center}
{
\scalebox{.8}{
\footnotesize{
	\begin{tikzpicture}[very thick]
			\node[state, inner sep =4pt] (110)  {$110$};
				\node[state, inner sep =4pt] (010) [right = 1.5 cm of 110] {$010$};
				\node[state, inner sep =4pt] (000) [right = 1.5 cm of 010] {$000$};
				\node[state, inner sep =4pt] (100) [right = 1.5 cm of 000] {$100$};
				\node[state, inner sep =4pt] (101) [right = 1.5 cm of 100] {$101$};
				\node[state, inner sep =4pt] (001) [right = 1.5 cm of 101] {$001$};		
				\node[state, inner sep =4pt] (011) [right = 1.5 cm of 001] {$011$};
				\node[state, inner sep =4pt] (111) [right = 1.5 cm of 011] {$111$};							
			\path
				(110)	edge[red,<->]	node[above=.1cm]{\(\)}	(010)			
				(000)	edge[red,<->]	node[above=.1cm]{\(\)}	(100)
				(101)	edge[red,<->]	node[above=.1cm]{\(\)}	(001)									
				(011)	edge[red,<->]	node[above=.1cm]{\(\)}	(111)
				(010)	edge[yellow,bend left,<->]	node[above=.1cm]{\(\)} (000)									
				(100)	edge[yellow,bend left,<->]	node[above=.1cm]{\(\)} (101)			
				(001)	edge[yellow,bend left,<->]	node[above=.1cm]{\(\)} (011)											
				(110)	edge[yellow,loop below,->]	node[below=.1cm]{\(\)} (110)	
				(111)	edge[yellow,loop below,->]	node[below=.1cm]{\(b\)} (111)		
				(010)	edge[cyan,bend right,<->]	node[above=.1cm]{\(\)} (000)									
				(100)	edge[cyan,loop above,->]	node[above=.1cm]{\(\)} (100)
				(101)	edge[cyan,loop above,->]	node[above=.1cm]{\(\)} (101)		
				(001)	edge[cyan,bend right,<->]	node[above=.1cm]{\(\)} (011)											
				(110)	edge[cyan,loop left,->]	node[left=.1cm]{\(a\)} (110)	
				(111)	edge[cyan,loop right,->]	node[right=.1cm]{\(c\)} (111)
				(010)	edge[green,loop above,->]	node[above=.1cm]{\(\)} (010)
				(000)   edge[green,loop above,->]	node[above=.1cm]{\(\)} (000)		
				(100)	edge[green,bend right,<->]	node[above=.1cm]{\(\)} (101)			
				(001)	edge[green,loop above,->]	node[above=.1cm]{\(\)} (001)
				(011)	edge[green,loop above,->]	node[above=.1cm]{\(\)} (011)
				(110)	edge[green,loop above,->]	node[above=.1cm]{\(a\)} (110)	
				(111)	edge[green,loop above,->]	node[above=.1cm]{\(d\)}(111)

		;
	\end{tikzpicture}
	}

	}	
}

\end{center}
\caption{The  Schreier graph on level 3 of the automata generating the first Grigorchuk group.}\label{fig:Gri3}
\end{figure}

\begin{figure}[h!]

\begin{center}

{
\scalebox{1.125}{
\footnotesize{
	\begin{tikzpicture}[->,>=latex,node distance=15mm]

		\node[state] (c)  {$c$};
				\node[state] (a) [left of=c,below of=c] {$a$};
				\node[state] (d) [left of=c, above of=c] {$d$};
				\node[state] (b) [right of=c, above of=c] {$b$};
				\node[state] (e) [right of=c,below of=c] {$e$};
			\path
				(b) edge[loop above] node[above]{\(\begin{array}{c}  1|1 \end{array}\)}	(b)
				(c) edge[loop above] node[above]{\(\begin{array}{c}  1|1 \end{array}\)}	(c)				
				(d) edge[loop above] node[above]{\(\begin{array}{c}  1|1 \end{array}\)}	(d)				
				(a)	edge	node[below=.1cm]{\(\begin{array}{c} 1|2 \quad 2|1  \quad  3|4 \\ 4|3  \quad  5|6  \quad 6|5 \end{array} \)}	(e)
				(d) edge 	node[left=.1cm]{\(\begin{array}{c} 2|2 \\3|3 \\ 4|5 \\ 5|4\\ 6|6 \end{array}\)} (a)
				(c) edge 	node[above=.1cm, pos=0.6]{\(\begin{array}{c} 2|3 \\3|2 \\ 4|4 \\ 5|5\\ 6|6 \end{array}\)} (a)
				(b) edge 	node[right=.1cm]{\(\begin{array}{c} 2|3 \\3|2 \\ 4|5 \\ 5|4\\ 6|6 \end{array}\)} (e)
				(e) edge[min distance=10mm,out=00,in=270,looseness=10] node[below right =-1mm and -1mm]{\(\begin{array}{c} id \end{array}\)}	(e)
		
		;
	\end{tikzpicture}
	}
	}	
}
{
\scalebox{.8}{
\footnotesize{
	\begin{tikzpicture}[very thick]
			\node[state, inner sep =4pt] (110)  {$1$};
				\node[state, inner sep =4pt] (010) [right = 1.5 cm of 110] {$2$};
				\node[state, inner sep =4pt] (000) [right = 1.5 cm of 010] {$3$};
				\node[state, inner sep =4pt] (100) [right = 1.5 cm of 000] {$4$};
				\node[state, inner sep =4pt] (101) [right = 1.5 cm of 100] {$5$};
				\node[state, inner sep =4pt] (001) [right = 1.5 cm of 101] {$6$};		
				
			\path
				(110)	edge[red,<->]	node[above=.1cm]{\(\)}	(010)			
				(000)	edge[red,<->]	node[above=.1cm]{\(\)}	(100)
				(101)	edge[red,<->]	node[above=.1cm]{\(\)}	(001)									

				(010)	edge[yellow,bend left,<->]	node[above=.1cm]{\(\)} (000)									
				(100)	edge[yellow,bend left,<->]	node[above=.1cm]{\(\)} (101)			
				(110)	edge[yellow,loop below,->]	node[below=.1cm]{\(\)} (110)	
				(001)	edge[yellow,loop below,->]	node[below=.1cm]{\(b\)} (001)		
				(010)	edge[cyan,bend right,<->]	node[above=.1cm]{\(\)} (000)									
				(100)	edge[cyan,loop above,->]	node[above=.1cm]{\(\)} (100)
				(101)	edge[cyan,loop above,->]	node[above=.1cm]{\(\)} (101)		
				(110)	edge[cyan,loop left,->]	node[left=.1cm]{\(a\)} (110)	
				(001)	edge[cyan,loop right,->]	node[right=.1cm]{\(c\)} (001)
				(010)	edge[green,loop above,->]	node[above=.1cm]{\(\)} (010)
				(000)   edge[green,loop above,->]	node[above=.1cm]{\(\)} (000)		
				(100)	edge[green,bend right,<->]	node[above=.1cm]{\(\)} (101)			
				(001)	edge[green,loop above,->]	node[above=.1cm]{\(\)} (001)
				(110)	edge[green,loop above,->]	node[above=.1cm]{\(a\)} (110)	
				(001)	edge[green,loop above,->]	node[above=.1cm]{\(d\)}(001)

		;
	\end{tikzpicture}
	}

	}	
}

\end{center}
\caption{A novel automata generating a group of intermediate growth (on the top) and its Schreier graph on level 1 (on the bottom). Notice the similarities with the Grigorchuk group on Fig.~\ref{fig:Mealy}.}\label{fig:T1}
\end{figure}
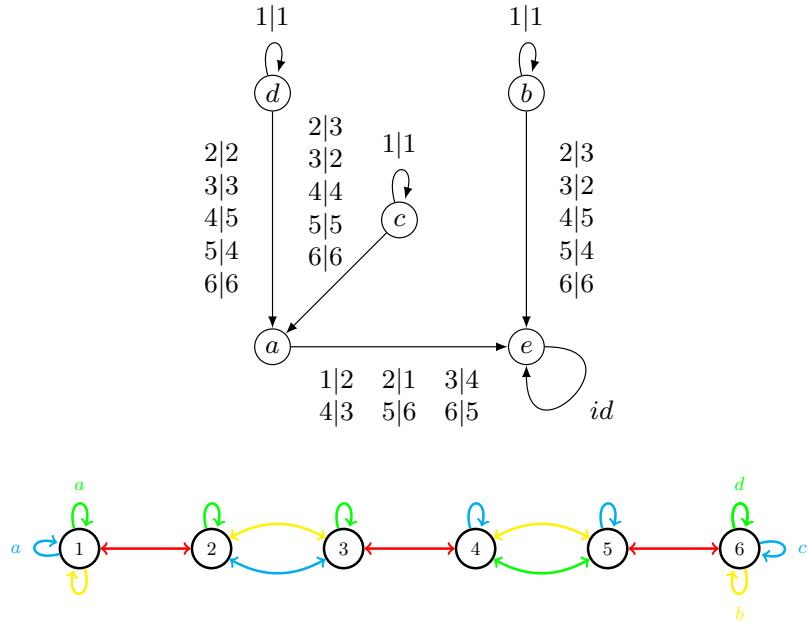

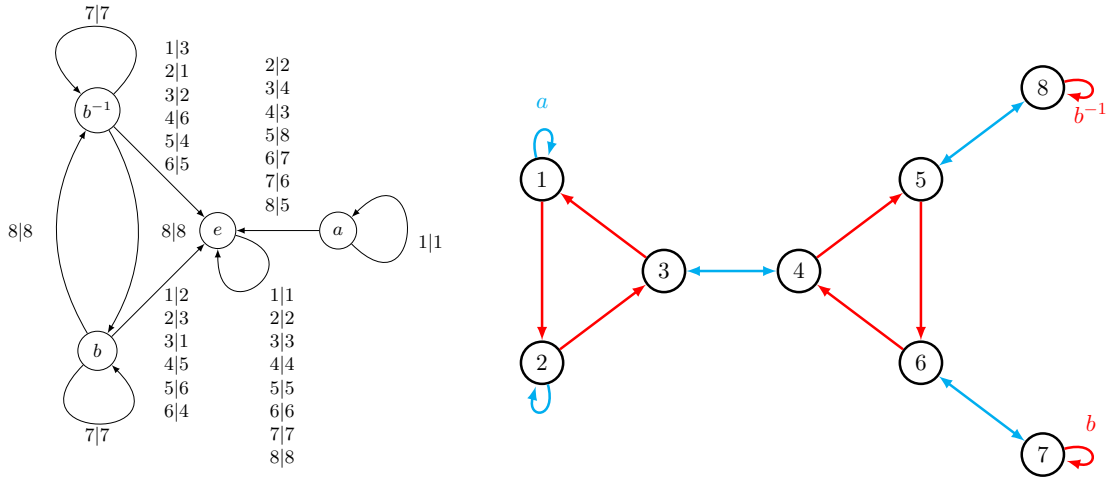
\begin{figure}[p]
\label{fig_automata}
\begin{center}

{
\scalebox{.8}{
\footnotesize{
	\begin{tikzpicture}[->,>=latex,node distance=20mm]

		\node[state, inner sep =4pt] (e)  {$e$};
				\node[state, inner sep =4pt] (b) [left of=c,below of=c] {$b$};
				\node[state, inner sep =2pt] (bi) [left of=c, above of=c] {$b^{-1}$};
				\node[state, inner sep =4pt] (a) [right of=c] {$a$};
			\path
		
				(a)	edge	node[above=.1cm]{\(\begin{array}{c} {2|2} \\ {3|4}\\{4|3}\\{5|8}\\{6|7}\\{7|6}\\{8|5} \end{array}\)}	(e)
				(b) edge 	node[below=.1cm, pos=0.7]{\(\begin{array}{c} {1|2}\\ {2|3} \\ {3|1}\\{4|5}\\{5|6}\\{6|4} \end{array}\)}(e)
				(bi) edge 	node[above=.1cm, pos=0.7]{\(\begin{array}{c} {1|3}\\ {2|1} \\ {3|2}\\{4|6}\\{5|4}\\{6|5} \end{array}\)}(e)

				(a) edge[min distance=10mm,out=-45,in=45,looseness=10] node[below right =-1mm and -1mm]{\(\begin{array}{c} {1|1} \end{array}\)}	(a)
				
				(b) edge[min distance=-10mm,out=225,in=315,looseness=10] node[below =-1mm ]{\(\begin{array}{c} {7|7} \end{array}\)}	(b)	
				(bi) edge[min distance=10mm,out=45,in=135,looseness=10] node[above =-1mm]{\(\begin{array}{c} {7|7} \end{array}\)}	(bi)			
				
				(b) edge[bend left] 	node[left=.1cm]{\(\begin{array}{c} {8|8}\end{array}\)}(bi)
				(bi) edge[bend left] 	node[right=.1cm]{\(\begin{array}{c} {8|8}\end{array}\)}(b)						

				(e) edge[min distance=10mm,out=-15,in=270,looseness=10] node[below right =-1mm and -1mm]{\(\begin{array}{c} {1|1}\\{2|2} \\ {3|3}\\{4|4}\\{5|5}\\{6|6}\\{7|7}\\{8|8} \end{array}\)}	(e)
		
		;
	\end{tikzpicture}
	}
	}	
}
{
\scalebox{.8}{
{
	\begin{tikzpicture}[->,>=latex,node distance=30mm, very thick]
		\node[state, inner sep =4pt] (3)  {$3$};
				\node[state, inner sep =4pt] (1) [above left = 1 cm and 1.5 cm of 3] {$1$};
				\node[state, inner sep =4pt] (2) [below left =  1 cm and 1.5 cm of 3] {$2$};
				\node[state, inner sep =4pt] (4) [right = 1.5 cm of 3] {$4$};
				\node[state, inner sep =4pt] (5) [above right =  1 cm and 1.5 cm of 4] {$5$};
				\node[state, inner sep =4pt] (6) [below right= 1 cm and 1.5 cm  of 4] {$6$};		
				\node[state, inner sep =4pt] (7) [below right= 1 cm and 1.5 cm  of 6] {$7$};
				\node[state, inner sep =4pt] (8) [above right=  1 cm and 1.5 cm of 5] {$8$};							
			\path
			
				(1)	edge[loop above,cyan]	node[above=.1cm]{\(\begin{array}{c} a \end{array}\)}	(1)			
			
				(2)	edge[loop below,cyan]	node[above=.1cm]{\(\begin{array}{c}  \end{array}\)}	(2)			
				
				(3)	edge[cyan,<->]	node[above=.1cm]{\(\begin{array}{c}  \end{array}\)}	(4)			
				(5)	edge[cyan,<->]	node[above=.1cm]{\(\begin{array}{c}  \end{array}\)}	(8)			
				(6)	edge[cyan,<->]	node[above=.1cm]{\(\begin{array}{c}  \end{array}\)}	(7)			
							
				(1)	edge[red]	node[left=.1cm]{\(\begin{array}{c}  \end{array}\)}	(2)
				(2)	edge[red]	node[below=.1cm]{\(\begin{array}{c}  \end{array}\)}	(3)
				(3)	edge[red]	node[above=.1cm]{\(\begin{array}{c}\end{array}\)}	(1)		
				
				(4)	edge[red]	node[right=.1cm]{\(\begin{array}{c}   \end{array}\)}	(5)
				(5)	edge[red]	node[below=.1cm]{\(\begin{array}{c}   \end{array}\)}	(6)
				(6)	edge[red]	node[above=.1cm]{\(\begin{array}{c}  \end{array}\)}	(4)			
				
				(7)	edge[loop right,red]	node[above=.1cm]{\(\begin{array}{c} b \end{array}\)}	(7)																			
				(8)	edge[loop right,red]	node[below=.1cm]{\(\begin{array}{c} b^{-1} \end{array}\)}	(8)									
		;
	\end{tikzpicture}
	}

	}	
}

\end{center}
\caption{The automaton~\({a=\wrpr{a,e,e,e,e,e,e,e}{(3,4)(5,8)(6,7)}} \:;\: {b=\wrpr{e,e,e,e,e,e,b,b^{-1}}{(1,2,3)(4,5,6)}}\) (on the left) from~\cite{Bri17} and  its  Schreier graph on level 1 (on the right).\label{fig:mNote}}
\end{figure}

\begin{figure}[p]

\begin{center}

{
\scalebox{1.25}{
\footnotesize{
	\begin{tikzpicture}[->,>=latex,node distance=15mm]

				\node[state] (c)  {$c$};
				\node[state] (a) [below left = .5cm and 1.5cm of c] {$a$};
				\node[state] (b) [below right = .5cm and 1.5cm of c] {$b$};
				\node[state] (e) [below = 3cm of c] {$e$};
			\path
				(a) edge[loop above] node[above]{\footnotesize \(\begin{array}{c}  1|1 \\ 2|2 \end{array}\)}	(a)			
				(b) edge[loop above] node[above]{\footnotesize \(\begin{array}{c}  5|5 \end{array}\)}	(b)
				(c) edge[loop right] node[above]{\footnotesize \(\begin{array}{c}  7|7 \end{array}\)}	(c)				
				(a)	edge	node[below left=-.2cm  and .3cm]{\footnotesize\(\begin{array}{c}  3|4 \\ 4|3  \\  5|5  \\ 6|7 \\ 7|6 \end{array} \)}	(e)
				(c) edge 	node[left=.1cm, pos=0.2]{\footnotesize\(\begin{array}{c} 1|2\\2|1 \\3|6 \\ 4|4 \\ 6|3 \end{array}\)} (e)
				(b) edge 	node[below right=-.2cm  and .3cm]{\footnotesize \(\begin{array}{c} 1|1\\2|3 \\3|2 \\ 4|5 \\ 5|4\\ 6|6 \end{array}\)} (e)
				(e) edge[min distance=10mm,out=-60,in=240,looseness=10] node[below =-1mm]{\(\begin{array}{c}id \end{array}\)}	(e)
		
		;
	\end{tikzpicture}
	}
	}	
}
{
\scalebox{1}{
\footnotesize{
	\begin{tikzpicture}[very thick]
				\node[state] (3)  {$3$};
				\node[state] (2) [above = 1cm of 3] {$2$};
				\node[state] (1) [above = 1cm of 2] {$1$};
				\node[state] (4) [below left = of c] {$4$};
				\node[state] (5) [below left = of 4] {$5$};
				\node[state] (6) [below right = of 3] {$6$};
				\node[state] (7) [below right = of 6] {$7$};				
				
			\path
				(3)	edge[red,<->]	node[above=.1cm]{\(\)}	(4)			
				(6)	edge[red,<->]	node[above=.1cm]{\(\)}	(7)
				(1)	edge[red,loop right,->]	node[right=.1cm]{\(a\)}	(1)									
				(2)	edge[red,loop right,->]	node[right=.1cm]{\(a\)}	(2)		
				(5)	edge[red,loop right,->]	node[above=.1cm]{\(\)}	(5)																				
				(3)	edge[yellow,<->]	node[above=.1cm]{\(\)}	(6)			
				(1)	edge[yellow,<->]	node[above=.1cm]{\(\)}	(2)
				(4)	edge[yellow,loop above,->]	node[above=.1cm]{\(\)}	(4)	
				(5)	edge[yellow,loop above,->]	node[above=.1cm]{\(b\)} (5)																					
				(7)	edge[yellow,loop above,->]	node[above=.1cm]{\(\)}	(7)						

				(3)	edge[cyan,<->]	node[above=.1cm]{\(\)}	(2)			
				(4)	edge[cyan,<->]	node[above=.1cm]{\(\)}	(5)
				(1)	edge[cyan,loop left,->]	node[above=.1cm]{\(\)}	(1)	
				(6)	edge[cyan,loop left,->]	node[above=.1cm]{\(\)} (6)																					
				(7)	edge[cyan,loop left,->]	node[left=.1cm]{\(c\)}	(7)

		;
	\end{tikzpicture}
	}

	}	
}

\end{center}
\caption{A novel automata generating a group of intermediate growth (on the left) and its Schreier graph on level 1 (on the right).}\label{fig:Y}
\end{figure}
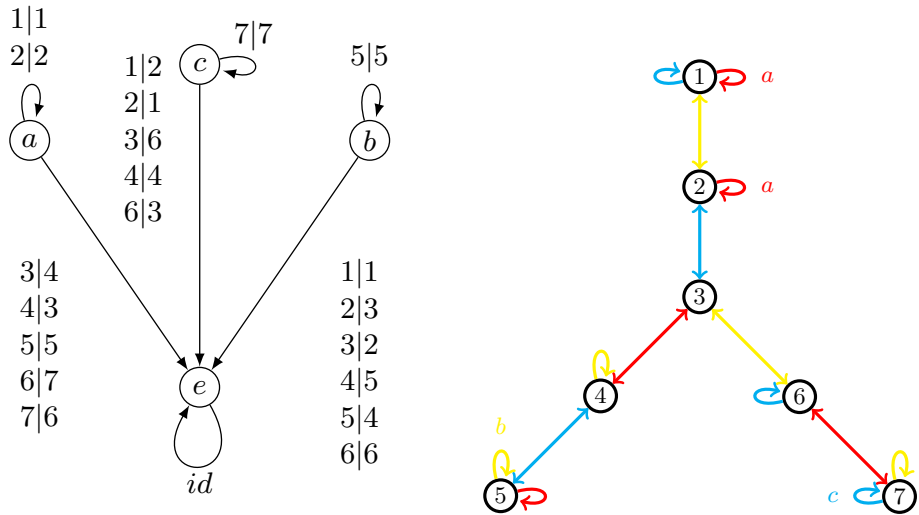
\end{document}